\documentclass[11pt,reqno,a4paper]{amsart}
\usepackage{graphicx}
\usepackage{amssymb,amsmath}
\usepackage{amsthm}
\usepackage{color,graphicx}
\usepackage{hyperref}
\usepackage{color}
\usepackage{mathabx}
\usepackage{subfigure}
\usepackage{appendix}

\usepackage{verbatim}

\usepackage{mathrsfs}

\usepackage{relsize} %tornar simbolos matemáticos maiores (usando \mathlarger{} )

\usepackage{mathtools} %para usar o comando \xrightharpoonup{\quad \:}

\newcommand{\be}{\begin{equation}}

\newcommand{\ee}{\end{equation}}

\newcommand{\cn}{{\rm \,cn}}
\newcommand{\sn}{{\rm \,sn}}
\newcommand{\dn}{{\rm \,dn}}

\newcommand{\Ker}{{\rm \,Ker}}
\newcommand{\K}{{\rm \,K}}
\newcommand{\E}{{\rm \,E}}

\def\smath#1{\text{\scalebox{.8}{$#1$}}}

\def\sfrac#1#2{\smath{\frac{#1}{#2}}}           %usando \sfrac{}

\numberwithin{equation}{section}
\numberwithin{figure}{section}

\newtheorem{theorem}{Theorem}[section]
\newtheorem{proposition}[theorem]{Proposition}
\newtheorem{remark}[theorem]{Remark}
\newtheorem{lemma}[theorem]{Lemma}

\newtheorem{definition}[theorem]{Definition}

\begin{document}
\vglue-1cm \hskip1cm
\title[abcd Boussinesq System]{Orbital stability of periodic traveling waves for the ``abcd" Boussinesq systems}

\begin{center}



\author[G.E.B. Moraes]{Gabriel E. Bittencourt Moraes}

\address{Gabriel E. Bittencourt Moraes - State University of
	Maring\'a, Maring\'a, PR, Brazil.}
\email{pg54546@uem.br}

\author[G. de Loreno]{Guilherme de Loreno}

\address{Guilherme de Loreno - State University of
	Maring\'a, Maring\'a, PR, Brazil.}
\email{pg54136@uem.br}

%\author[G.M. Muslu]{Gulcin M. Muslu}

\author[F. Natali]{F\'abio Natali}

\address{F\'abio Natali - Department of Mathematics, State University of
	Maring\'a, Maring\'a, PR, Brazil. }
\email{fmanatali@uem.br }

\maketitle

\vspace{1mm}

\end{center}

\begin{abstract}
New results concerning the orbital stability of
periodic traveling wave solutions for the ``abcd" Boussinesq  model will be shown in this manuscript. For the existence of solutions, we use basic tools of ordinary differential equations to show that the corresponding periodic wave depends on the Jacobi elliptic function of cnoidal type. The spectral analysis for the associated linearized operator is determined by using some tools concerning the Floquet theory. The orbital stability is then established by
applying the abstract results \cite{andrade} and \cite{GrillakisShatahStraussI} which give us sufficient conditions to the orbital stability for a general class of evolution equations.
\end{abstract}

\section{Introduction}\label{Introduction}

In what follows, consider the well known ``abcd'' Boussinesq  system
\begin{equation}\label{abcd}\left\{\begin{array}{llll}
\eta_t+u_x+(\eta u)_x+a u_{xxx}-b\eta_{xxt}=0 \\
u_t+\eta_x+uu_x+c\eta_{xxx}-du_{xxt} = 0,
\end{array}\right.
\end{equation}
where $\eta, u:\mathbb{R}\times (0, +\infty) \longrightarrow\mathbb{R}$ are real-valued functions which are $L$-periodic at the spatial variable and $a,b,c,d \in \mathbb{R}$ are suitable real parameters. The ``abcd'' model $(\ref{abcd})$ was introduced by Bona, Chen and Saut  \cite{BonaChenSautI} and \cite{BonaChenSautII} to describe the motion of small-amplitude long waves on the surface of an ideal fluid under the force of gravity. The quantity $\eta(x,t)$ is the vertical deviation of the free surface from its rest position while $u(x,t)$ represents the horizontal velocity field at time $t>0$. Initially, the constants $a,b,c,d$ must satisfy only the following relation 
\begin{equation}\label{relation-abcd-tau}
	a + b + c + d = \frac{1}{3} - \tau,
\end{equation} 
where $\tau \geq 0 $ is the surface tension of fluid. As reported in \cite{BonaChenSautI}, when the surface tension $\tau$ is zero, parameters $a,b,c,d $ must satisfy the relations
\begin{align}\label{relationabcd1}
& a+b=\frac{1}{2}\left( \theta^2-\frac{1}{3}\right),  \nonumber \\
& c+d=\frac{1}{2}(1-\theta^2) \geq 0,  \\
& a+b+c+d=\frac{1}{3} \nonumber , 
\end{align}
where $\theta \in [0,1]$. In addition, $a,b,c,d$ can be rewritten in the form
\begin{equation}\label{relationabcd2}
\begin{array}{lcl} \displaystyle{a=\frac{1}{2}\left(\theta^2-\frac{1}{3}\right)\nu,} & & \displaystyle{b=\frac{1}{2}\left(\theta^2-\frac{1}{3}\right)(1-\nu)},\\
\displaystyle{c=\frac{1}{2}(1-\theta^2)\,\mu}, & &  \displaystyle{d=\frac{1}{2}(1-\theta^2)\,(1-\mu)},
\end{array}
\end{equation}
where $\nu, \mu$ are suitable real parameters in the sense that \eqref{relationabcd2} implies \eqref{relationabcd1}. In \cite{BonaChenSautI}, the authors investigated the parameters $a,b,c,d$ in details.  Indeed, by choosing different real values for $\nu, \mu$ and $\theta \in [0,1]$, it is possible to deduce some classical systems in the current literature as the classical Boussinesq system, Kaup system, Bona-Smith system, coupled Benjamin-Bona-Mahony system, coupled Korteweg-de Vries system, coupled mixed Korteweg-de Vries-Benjamin-Bona-Mahony systems, and many others. Additional aspects concerning all mentioned models above can be seen in \cite{Kaup}, \cite{Sachs}, and \cite{Winther}.

Let $L>0$ be fixed and consider $(\eta, u):=\begin{pmatrix}
	\eta \\ 
	u
\end{pmatrix} \in H^1_{per} \times H^1_{per}$. By considering $b=d$ in \eqref{abcd}, one has the conserved quantities  $E,F:H^1_{per}  \times H^1_{per} \rightarrow \mathbb{R}$ given by
\begin{equation}\label{E}
	E(\eta,u)=\int_{0}^{L} -c\eta_x^2-au_x^2+\eta^2+(1+\eta)u^2  \; dx
\end{equation}
and
\begin{equation}\label{F}
	F(\eta,u)=\int_{0}^{L} \eta u+b \, \eta_x u_x  \; dx.
\end{equation}
Moreover, $(\ref{abcd})$ also admits the following conserved quantities
\begin{equation}\label{M1M2}
M_1(\eta, u)= \int_0^L \eta \; dx \quad \text{and} \quad M_2(\eta, u)= \int_0^L u \; dx.
\end{equation} 
\indent Another important fact is that \eqref{abcd} can be written as an abstract Hamiltonian system
\begin{equation}\label{hamilt}
\frac{d}{dt}\begin{pmatrix}
\eta \\ 
u
\end{pmatrix} = J E'\begin{pmatrix}
\eta \\ 
u
\end{pmatrix},
\end{equation}
where
\begin{equation}\label{J}
J=\begin{pmatrix}
0 & (1-b\partial_x^2)^{-1}\partial_x \\ 
 (1-b\partial_x^2)^{-1}\partial_x & 0
\end{pmatrix}
\end{equation}
is a skew-symmetric operator defined in $L_{per}^2\times L_{per}^2$ with domain $ H^1_{per} \times H^1_{per}$.\\
\indent Let $\omega \in \mathbb{R} $ be fixed. Consider the Lyapunov functional depending on the conserved quantities $E$, $F$, $M_1$ and $M_2$ defined as
\begin{equation}\label{G}
	G(\eta, u)=E(\eta, u)-\omega F(\eta, u)-A_2 M_1(\eta, u)-A_1 M_2(\eta,u),
\end{equation}
where $A_1$ and $A_2$ are the constants of integration (see $(\ref{SystemEDO2})$). As far as we know, one of the requirements to determine the orbital stability is to obtain a pair $(\varphi,\psi)$ of special solutions called \textit{traveling waves} for $(\ref{abcd})$ satisfying $G'(\varphi,\psi)=0$, that is, $(\varphi,\psi)$ needs to be a critical point of $G$ (see \cite{bona2}). In addition, the linearized operator around the pair $(\varphi,\psi)$ is given by 
\begin{equation}\label{matrixop}
	\mathcal{L}= G''(\varphi,\psi)=\begin{pmatrix}
		1+c\partial_x^2 & b\omega \partial_x^2+\psi-\omega \\
		b \omega \partial_x^2+\psi-\omega &1+a\partial_x^2+\varphi
	\end{pmatrix}.
\end{equation}
The operator $\mathcal{L}$ is a self-adjoint, deﬁned in $L^2_{per} \times L^2_{per}$ with dense domain $H^{2}_{per} \times H^{2}_{per}$. As we will see later on, operator $\mathcal{L}$ in $(\ref{matrixop})$ plays an important role in our study.\\
\indent Next, we present some traveling waves associated to the system $(\ref{abcd})$. In fact, the existence of suitable conserved quantities for the model $(\ref{abcd})$ suggests in the most of the cases, the existence of traveling wave solutions of the form,
\begin{equation}\label{standingwave}
(\eta(x,t), u(x,t))=\big(\varphi(x-\omega t), \psi(x-\omega t)\big),
\end{equation}
where $\varphi, \psi: \mathbb{R} \longrightarrow \mathbb{R}$ are smooth $L$-periodic functions and $\omega \in \mathbb{R}$ represents the wave speed. Substituting \eqref{standingwave} into \eqref{abcd}, we obtain the following system of ordinary differential equation 
\begin{equation}\label{SystemEDO1}
\begin{cases}
-\omega \varphi'+\psi'+(\varphi \psi)'+a \psi'''+b\omega \varphi'''=0 \\
-\omega \psi'+\varphi'+\frac{1}{2}(\psi^2)'+c\varphi'''+d\omega \psi'''=0,
\end{cases}
\end{equation}
or equivalently, after an integration in both equations
\begin{equation}\label{SystemEDO2}
\begin{cases}
-\omega \varphi+\psi+\varphi \psi+a \psi''+b\omega \varphi''-A_1=0 \\
-\omega \psi+\varphi+\frac{1}{2}(\psi^2)+c\varphi''+d\omega \psi''-A_2=0,
\end{cases}
\end{equation}
where $A_1, A_2$ are real constants of integration to be chosen later.\\
\indent Our inspiration to show the existence of a pair of solution $(\varphi, \psi)$ for the system \eqref{SystemEDO2} was the seminal paper of Chen, Chen and Nguyen \cite{ChenChenNguyen} (see also \cite{MutoChristiansenLomdahl}) where the authors showed the existence of periodic traveling wave solutions $(\varphi,\psi)$ of cnoidal type $(\varphi, \psi)$ for the system \eqref{abcd} for the case $a=c=0$ and $b=d=\frac{1}{6}$. They put forwarded that $\varphi$ and $\psi$ are given explicitly by
\begin{equation}\label{LIcn} \varphi(\xi) = \gamma_0 + \gamma_2 {\rm cn}^2(\lambda \xi, k) + \gamma_4 {\rm cn}^4(\lambda \xi, k) \quad \text{ and } \quad \psi(\xi) = \zeta_0 + \zeta_2 {\rm cn}^2(\lambda \xi, k),
\end{equation}
for some real constants $\gamma_i$ and $\zeta_j$, $i=0,2,4$ and $j=0,2$. Parameter $k\in(0,1)$ in $(\ref{LIcn})$ is called modulus of the elliptic function. Solutions in $(\ref{LIcn})$ are non-multiple each other and this fact brings additional difficulties in order to study the behaviour of the non-positive spectrum of $\mathcal{L}$ defined in $(\ref{matrixop})$. To overcome this difficulty concerning the behaviour of the non-positive spectrum of $\mathcal{L}$, we are going to assume that $(\varphi, \psi) = (\varphi, B \varphi)$, where $B$ is a convenient constant. Consider $a=c<0$ and $b=d>0$. Substituting the anstaz $(\varphi, \psi)=(\varphi, B \varphi)$ into the system \eqref{SystemEDO2} with $\varphi$ given by
\begin{equation}\label{solcn1}
\varphi(x)=b_0+b_2 \, {\rm cn}^2\left( \frac{2{\rm K}(k)}{L} x, k \right),
\end{equation}
we can see, by a convenient choose of parameters $b_0$, $b_2$, $B$, $A_1$, and $A_2$ that $(\varphi, B\varphi)$ is a solution of \eqref{SystemEDO2}. In fact, we need to consider two basic cases in our study. The first one is to assume $a+b=0$ to obtain the existence of a smooth surface 
\begin{equation}\label{smooth1}
(b,\omega,k) \in (0,+\infty)\times(-1,1)\times \Big(0,\tfrac{1}{\sqrt{2}} \Big)\subset \mathbb{R}^3\mapsto (\varphi,  \psi)=(\varphi_{b,\omega,k}, B \varphi_{b,\omega,k}) \in H^2_{per} \times H^2_{per} 
\end{equation} 
of periodic negative solutions, where $\varphi$ is given by $(\ref{solcn1})$. In that case, parameters $b_0,b_2$, $A_1$, $A_2$ and $B$ are given explicitly by
\begin{equation}\label{b0case11}
	b_0=\frac{-64\,(B\omega - 1)^2 \,b\, \left(k^2 - \frac{1}{2}\right){\rm K}(k)^2 + 2\left( \left(\omega^2 + \frac{1}{2}\right)B - \frac{3}{2}{\omega}\right) B L^2}{2 \, B^2 \, L^2 (B\omega - 1)}
\end{equation}
and
\begin{equation}\label{b2case11}
	b_2=\frac{48\,b \, k^2 \, (B\omega -1) {\rm K}(k)^2}{B^2 \, L^2},
\end{equation}
where $k \in \big(0, \sfrac{1}{\sqrt{2}}\big)$ and ${\rm K}(k)$ indicates the complete integral elliptic of the first kind. The wave speed $\omega$ is a free parameter belonging to the open interval $(-1,1)$ and parameter $B$ can be expressed into two different ways, namely,
\begin{equation}\label{Bcase11}
	B=-\frac{\omega}{2}+\frac{1}{2}\sqrt{\omega^2+8}>0 \quad \text{or} \quad B=-\frac{\omega}{2}-\frac{1}{2}\sqrt{\omega^2+8}<0.
\end{equation}
In addition, the integration constants $A_1, A_2$ are defined as
\begin{equation}\label{A1case11}
	A_1=-{\frac { \left( B-\omega \right) ^{2} ( -256\,{b}^{2} \, r(k) \,  {\rm K}(k)^{4}+{L}^{4} )}{4\,B{L}^{4}}}
\end{equation}
and
\begin{equation}\label{A2case11}
	A_2={\frac {1024\,{b}^{2} r(k)  \left( B\omega
			-1 \right) ^{4}  {\rm K} (k)^{4}-
			4\left( {B}^{3} s(\omega)+ \left( -3\,{\omega}
			^{3}- \tfrac{3\omega}{2}\right) {B}^{2}+ \left( \tfrac{11}{4}{\omega}^{2}+1 \right) B-\omega
			\right)B{L}^{4}}{8\,{L}^{4}{B}^{2} \left( B\omega-1 \right)^{2}}},
\end{equation}
where $r(k)= {k}^{4}-{k}^{2}+1$ and $s(\omega)= {\omega}^{4}+{\omega}^{2}-\tfrac{1}{4}$.

 The second one is to assume  $a+b = \frac{1}{6}$. We obtain an explicit pair of solutions $(\varphi, B\varphi)$ for the same system with $\omega=0$. In this case, we  initially determine a smooth surface
\begin{equation}\label{smooth2}
(a,k) \in \left(-\infty,0\right)\times (0,1)   \subset \mathbb{R}^2\longmapsto (\varphi,  \psi)=(\varphi_{a,k}, B \varphi_{a,k}) \in H^2_{per} \times H^2_{per} 
\end{equation}
of periodic solutions such that $(\varphi,B\varphi)$ satisfies \eqref{SystemEDO2}. Parameters $b_0$, $b_2$, $B$, $A_1$, and $A_2$ must satisfy
\begin{equation}\label{b0-case2-w=01}
	b_0 = - \frac{1}{2} \frac{(32 k^2 - 16)\, a \, {\rm K}(k)^2 + L^2}{L^2}
\end{equation}
and 
\begin{equation}\label{b2-case2-w=01}
	b_2 = \frac{24 {\rm K}(k)^2 \, a \, k^2}{L^2}.
\end{equation}

In this case, $B$ can be expressed into two different simple ways, namely, $B = \pm \sqrt{2}$. In addition, 
\begin{equation}\label{A1-case2-w=01}
	A_1 = \frac{1}{4} \frac{ \sqrt{2} \, (-256 \, {\rm K}(k)^4 \, a^2 \, r(k) + L^4)}{L^4}
\end{equation}
and
\begin{equation}%\label{A2-case2-w=0}
	A_2 = -\frac{1}{4} \frac{ (-256 \, {\rm K}(k)^4 \, a^2 \, r(k) + L^4)}{L^4}.
\end{equation}

The reason to consider the cases $b=d > 0$ and $a=c<0$ is two fold: the first one is to write $(\ref{abcd})$ as a Hamiltonian form given by $(\ref{hamilt})$ and the second reason is because the pair $(\varphi,B\varphi)$ can be expressed in terms of the cnoidal periodic solution as in $(\ref{solcn1})$. Depending on the choice $a+b=0$ or $a+b=\frac{1}{6}$, we can obtain the smooth surfaces $(\ref{smooth1})$ or $(\ref{smooth2})$, respectively. Both conditions concerning parameters $a,b,c,d$ were discussed by Hakkaev, Stanislavova and Stefanov \cite{HakkaevStanislavovaStefanov} to obtain  spectral stability results of solitary waves associated to the system \eqref{abcd}. More precisely, they consider traveling wave solutions $\varphi$ having a ${\rm sech}^2$ profile with $\psi=B \varphi$ and employed the arguments in \cite{KapitulaStefanov} to prove that the spectrum  of the corresponding linearized operator obtained by a suitable linearization around the pair $(\varphi,\psi)$ is located entirely on the imaginary axis.\\
\indent Concerning the standard Benjamin-Bona-Mahony equation (BBM henceforth),
\begin{equation}\label{BBM}
	u_t+u_x+uu_x-u_{xxt}=0,
	\end{equation}
and its analogous equation $(\ref{SystemEDO2})$ which determines traveling wave solutions of the form $u(x,t)=\psi(x-\omega t)$ given by
\begin{equation}\label{ode1}
	\omega\psi''-(\omega-1)\psi+\frac{1}{2}\psi^2-A=0,
	\end{equation}
precise results of orbital stability for cnoidal solutions associated to the equation $(\ref{ode1})$ with $A=0$ were determined in \cite{an-ban-scia1}. The authors proved that the associated linearized operator $\mathcal{L}$ has only one negative eigenvalue which is simple and zero is a simple eigenvalue whose associated eigenfunction is $\psi'$ by using the Fourier expansion of the explicit solutions of cnoidal type together with the main theorem in \cite{AnguloNatali2008}.
In addition, using the arguments in \cite{AnguloBonaScialom}, the explicit solutions were also used to calculate the positiveness of $\frac{d}{dc}\int_0^L(\psi'^2+\psi^2)dx$ in order to conclude the orbital stability.

\indent The well-posedness of the Cauchy problem for the system $(\ref{abcd})$ needs to be highlighted. In fact, most of contributions concerns the case where the system $(\ref{abcd})$ is considered in the whole real line instead of periodic boundary conditions. We first refer the reader the pioneers works by Bona, Chen and Saut \cite{BonaChenSautI} and  \cite{BonaChenSautII} where some sufficient conditions for the local well-posedness have been discussed in details (see also \cite{Angulo1999} for an additional reference). Global solutions in the periodic context have been determined in \cite{YangZhang} when the regularized term in $(\ref{abcd})$ is dropped, that is, when $b=d=0$. In some particular cases, the authors in \cite{BonaChenSautII} have been discussed the problem of global well-posedness with some restrictions on the initial data. It is also important to mention that most of theories of orbital stability for traveling waves can be applied if the associated Cauchy problem has a suitable global well-posedness result in the energy space, that is, the space where the energy functional is defined (in our case $E$ in $(\ref{E})$ is defined in $H_{per}^1\times H_{per}^1$). As examples, we can cite \cite{ANP}, \cite{bona2}, \cite{CristofaniPastor}, \cite{NataliPastorCristofani}, \cite{WeinsteinNLS}, and references therein. As far as we can see, our model does not have a suitable global well-posedness result in the energy space $H_{per}^1\times H_{per}^1$, except if additional conditions about the smallness of the initial data are assumed (see \cite{BonaChenSautII}). For our purposes, we are going to assume a suitable local well-posedness, but we believe that the desired result can be obtained by repeating similar arguments as in \cite{BonaChenSautI} and \cite{BonaChenSautII} since the approach is based upon standard fixed point arguments.\\
\indent The pioneer work by Grillakis, Shatah and Strauss \cite{GrillakisShatahStraussI} establishes orbital stability result by considering only a local solutions in time for the associated Cauchy problem in the energy space. Our results are based on this approach and one of its generalizations as in Andrade and Pastor \cite{andrade} where they consider the standard problem of minimization of the energy with a more general constraint. An important fact is that the authors in \cite{andrade} do not realize that their approach can be determined only assuming the existence of local solutions (global solutions have been assumed, in fact). The contradictory and sequence arguments as done in \cite{GrillakisShatahStraussI} avoid the assumption concerning the global well-posedness since it is possible to work only with a local time for the existence of solutions. Without further ado, we describe the main points to get the orbital stability. In fact, the main result in \cite{andrade} reads as follows. Suppose that:
\begin{itemize}
	\item [(H1)] there exists an $L$-periodic solution of \eqref{SystemEDO1}, say $\Psi = (\varphi,\psi) \in C^\infty_{per}([0,L]) \times C^{\infty}_{per}([0,L])$ with minimal period $L>0$. The self-adjoint operator $\mathcal{L}$ has only one negative eigenvalue, which is simple, and zero is a simple eigenvalue with associated eigenfunction \linebreak $\Psi' = (\varphi',\psi')$.
	
	\item [(H2)] There exists $\Phi \in X = H_{per}^1 \times H_{per}^1$ such that
	\begin{equation}
	I=	(\mathcal{L} \Phi, \Phi)_{L^2_{per} \times L^2_{per}} < 0 \quad   \text{ and }  \quad (\mathcal{L} \Phi, \Psi')_{L^2_{per} \times L^2_{per}} = 0,
	\end{equation}
	where $\mathcal{L} \Phi=Q'(\Psi) $ with $Q$ being a convenient conserved quantity $Q = Q(\eta, u)$.
\end{itemize}
Thus, $\Psi=(\varphi, \psi)$ is orbitally stable in $X$ by the periodic flow of \eqref{abcd}. 

\begin{remark}\label{rem123} An important fact to obtain the orbital stability is that a condition (H4) in \cite{andrade} needs to be satisfied regarding a relation among the conserved quantities of the model. However, we do not present here because it is no longer necessary. Indeed, as determined in \cite{andrade}, we do not parametrize our solutions in terms of the elliptic modulus $k\in(0,1)$ and we will show in Section $\ref{stability-section}$ that such condition can be easily verified.
	\end{remark}

The first condition in (H1) has been determined in $(\ref{smooth1})$ and $(\ref{smooth2})$. To obtain the second condition, we turn our attention to the spectral analysis of the operator $\mathcal{L}$ given by \eqref{matrixop}. Indeed, by \eqref{SystemEDO1} we see that $\mathcal{L}(\varphi', \psi')=(0,0)$, that is, $(\varphi', \psi') \in {\rm Ker}(\mathcal{L})$. The fact that solutions $\varphi$ and $\psi$ of $(\ref{SystemEDO2})$ are multiple each other can be used to obtain the existence of a similar transformation $T$ such that 
\begin{equation*}
\mathcal{L}=T^{-1} M T,
\end{equation*}
where $T=\begin{pmatrix}
	\tfrac{1}{\sqrt{2}} & \tfrac{1}{\sqrt{2}} \\ 
	-\tfrac{1}{\sqrt{2}} & \tfrac{1}{\sqrt{2}}
\end{pmatrix} $. The matrix operator $M$ is given by
\begin{equation*}
M=\begin{pmatrix}
-\partial_x^2(-a-b\omega)+(1-\omega)+\psi+\tfrac{\varphi}{2} & \tfrac{\varphi}{2} \\ 
\tfrac{\varphi}{2} & -\partial_x^2(-a+b\omega)+(1+\omega)-\psi+\tfrac{\varphi}{2}
\end{pmatrix}.
\end{equation*}
This fact allows us to obtain, in the case $a+b=0$, the following similar transformations
\begin{equation*}
	\mathcal{L} = (UST)^{-1} \left( 
	\begin{array}{cc}
		\mathcal{L}_1 & 0 \\
		0 & \mathcal{L}_2
	\end{array}
\right) (UST),
\end{equation*}
where $$U=\begin{pmatrix}
	\alpha
	& \beta
	\\ 
	\noalign{\medskip}{\frac {\alpha\, \left( 2\,B+\omega-3 \right) }{
			\sqrt {1-{\omega}^{2}}}}
	&  -{\frac {\beta\, \left( 2\,B+\omega+3 \right) }{
			\sqrt {1-{\omega}^{2}}}}
\end{pmatrix} \quad \mbox{and} \quad S = \left( \begin{array}{cc}
	\sqrt{b(1-\omega)} & 0 \\
	0 & \sqrt{b(1+\omega) }
\end{array}
\right), $$
with \begin{eqnarray*}
	 \alpha =  -\sqrt {b \left( 1+\omega \right) } \left( 2\,B+\omega+3 \right)  \left( B-1 \right)\quad 
\mbox{and} \quad
	\beta = -\sqrt {b \left( 1-\omega \right) } \left( 1+B \right)  \left( -2\,B-\omega+
	3 \right).
\end{eqnarray*}

 Matrix operators $\mathcal{L}_1$ and $\mathcal{L}_2$ are given by
\begin{equation}\label{L1L2case12}
	\mathcal{L}_1 = -\partial_x^2 + \frac{1}{b} + \ell_1 \varphi \quad \text{ and } \quad \mathcal{L}_2 = -\partial_x^2 + \frac{1}{b} + \ell_2 \varphi,
\end{equation}
where \begin{equation}\label{l1l20}
	\ell_1 = \frac{4 + 2\, B\,\omega}{2\, b \, (1-\omega^2)} \quad \text{ and } \quad \ell_2 = \frac{-2 + 2\, B\,\omega }{2\, b \, (1-\omega^2)}.
\end{equation}

If $a+b = \frac{1}{6}$, we have similar facts since $\mathcal{L}$ in $(\ref{matrixop})$ can be expressed by
\begin{equation*}
\mathcal{L}=(\mathcal{S}T)^{-1} \begin{pmatrix}
\mathcal{L}_3 & 0 \\ 
0 &  \mathcal{L}_4
\end{pmatrix}(\mathcal{S}T),
\end{equation*}
where $\mathcal{S}$ depends on the choice of the $B=\pm\sqrt{2}$. For instance, when $B<0$ we have $\mathcal{S}= \tfrac{1}{\sqrt{6}} \begin{pmatrix}
	\sqrt{3+2\sqrt{2}} & \sqrt{3-2\sqrt{2}} \\ 
	-\sqrt{3-2\sqrt{2}} & \sqrt{3+2\sqrt{2}}
\end{pmatrix}.$ The Hill operators $\mathcal{L}_3$ and $\mathcal{L}_4$ are given by
\begin{equation}\label{L1L2case21}
	\mathcal{L}_3=-a\partial_x^2+1+2\varphi \quad \text{and} \quad \mathcal{L}_4=-a\partial_x^2+1-\varphi.
\end{equation}
\indent Using Sylvester law of inertia, we can study the behaviour of the non-positive spectrum of $\mathcal{L}$ by studying only the non-positive spectrum of the Hill operators $\mathcal{L}_i$, $i\in \{1,2,3,4\}$, given by $(\ref{L1L2case12})$ and $(\ref{L1L2case21})$. To this end, we determine the quantity and multiplicity of non-positive eingenvalues for these Hill operators by combining some tools concerning the Floquet theory and some arguments contained in Natali and Neves \cite{NataliNeves2014} about the isoinertially of second order linear operators. In all cases, we are able to determine that the kernel of $\mathcal{L}$ is simple with ${\rm Ker}(\mathcal{L})=[(\varphi', \psi')]$ and the number of negative eigenvalues of $\mathcal{L}$ is one as requested in (H1).

 To prove (H2) we need to show the existence of $\Phi \in X$ such that $\mathcal{L}\Phi=Q'(\Psi)$, \linebreak $I = (\mathcal{L} \Phi, \Phi)_{L^2_{per} \times L^2_{per} }<0$ and $(\mathcal{L}\Phi,\Psi')_{L^2_{per}\times L^2_{per}}=0$ for both cases $a+b = 0$ and $a+b = \sfrac{1}{6}$. First, we consider the case $a+b =0$. By $(\ref{smooth1})$, one has the existence of a smooth surface of periodic waves $(\varphi,\psi)$ for the system $(\ref{SystemEDO2})$ depending on $(b,\omega,k)$, where $b>0$, $\omega \in (-1,1)$ and $k \in (0,\sfrac{1}{\sqrt{2 }})$. The smoothness in terms of $\omega$ enables us to consider $\Phi = (\partial_\omega \varphi , \partial_\omega \psi)$. Defining the conserved quantity $Q$ given by
 $
 	Q(\eta,u) = F(\eta,u) + (\partial_\omega A_2) M_1 (\eta,u) + (\partial_\omega A_1) M_2(\eta,u),
$
we see clearly that $\mathcal{L}\Phi=Q'(\Psi)$ and $(\mathcal{L}\Phi,\Psi')_{L^2_{per}\times L^2_{per}}=0$. To compute the quantity $I$, we need to perform hard calculations involving $b>0$, $\omega\in (-1,1)$ and $k\in(0,\sfrac{1}{\sqrt{2}}\big)$. To overcome this difficult, we employ some numerical computations using  \textit{software Maple} in order to obtain our desired result in the assumption (H2).\\
\indent We need to explain additional facts concerning the case $a+b = \frac{1}{6}$. As in \cite{HakkaevStanislavovaStefanov}, we do not have a curve with respect to the wave speed $\omega$ since we obtain the existence of cnoidal waves of the form $(\ref{cnoidalsolution})$ only for $\omega = 0$. Thus, we are forced to consider, in order to determine a  way to decide if $I$ is negative, that $A_1 = A_2 = 0$. This additional assumption enables us to get that constant $a$ in the system $(\ref{abcd})$ is given by
\begin{equation}\label{relationa-intro}
	a = -\frac{1}{16} \frac{L^2}{{\rm K}(k)^2 \sqrt{k^4-k^2+1}}.
\end{equation}
In this case, we obtain a smooth curve $\varphi=\varphi_k$ depending on $k \in (0,1)$ and since $a+b=\frac{1}{6}$, we need to consider the surface tension $\tau = 0$, so that constants $a,b,c,d \in \mathbb {R}$ must satisfy equalities in \eqref{relationabcd2}. %Thus, $\nu \in \mathbb{R}$ is given by
%$
%	\nu = - \frac{3}{8} \frac{L^2}{{\rm K}(k)^2 \sqrt{k^4-k^2+1}},
%$
%and it makes sense to consider the equality $a+b=\tfrac{1}{6}$ for the constant $a <0$ given in \eqref{relationa-intro}. Consequently, since $a+b=\frac{1}{6}$, we have that $b > 0$ becomes
%$
%	b = \frac{1}{6} + \frac{1}{16} \frac{L^2}{{\rm K}(k)^2 \sqrt{k^4-k^2+1}}.
%$ 
To obtain assumption (H2), consider the conserved quantity $Q(\eta,u) = F(\eta,u)$. 
If $\varphi$ has depended smoothly on the wave speed $\omega$, the fact that $A_1=A_2=0$ would imply that the natural choice of $\Phi\in X$ such that $\mathcal{L}\Phi=Q'(\Psi)$ could be $\Phi=(\partial_{\omega}\varphi,\partial_{\omega}\psi)$. However, we only have the pair of solutions $(\varphi,\psi)$ for the equation $(\ref{SystemEDO2})$ with $\varphi$ given by $(\ref{cnoidalsolution})$ and $\psi=B\varphi$ when $\omega=0$. However, since ${\rm Ker}(\mathcal{L})=[(\varphi', \psi')]$ we obtain $Q'(\Psi) \in {\rm Ker}(\mathcal{L})^{\tiny \perp}={\rm R}(\mathcal{L})$, so that there exists $\Phi \in {\rm D}(\mathcal{L})$ such that $\mathcal{L}\Phi=Q'(\Psi)$ but we do not know an explicit expression for this new $\Phi$. To overcome this difficulty, we establish a convenient estimate for $I=(\mathcal{L}\Phi,\Phi)_{L_{per}^2\times L_{per}^2}$ to obtain that $I<0$. Some numerics and suitable estimates are needed.\\
\indent Our results can be summarized in the next theorem.

\begin{theorem}\label{mainT} Let $L>0$ be fixed.\\
	{\rm (i)} Let $(\varphi, \psi) = (\varphi_{b,\omega,k}, B \varphi_{b,\omega,k})$ be the pair of periodic solution for the system $(\ref{SystemEDO2})$ given by $(\ref{smooth1})$. There exists $b^{*}>0$ depending on $L>0$ such that if $b>b^*$ then the wave $\Psi=(\varphi, B\varphi)$ is orbitally stable in $X$.\\
	{\rm (ii)} In system $(\ref{SystemEDO2})$, consider $A_1=A_2=0$. Let $(\varphi, \psi)=(\varphi_{k}, B\varphi_{k})$ be the periodic solution given by $(\ref{smooth2})$. The periodic wave $\Psi=(\varphi, \psi)$ is orbitally stable in $X$.
\end{theorem}

%In conclusion, according to (H1), (H2) and \cite{CristofaniPastor}, we are now in position to conclude our main result regarding orbital stability/instability that is the periodic wave $(\varphi, \psi)$ is orbitally stable for $\omega \in (-1,1)$, $k \in \big(0, \sfrac{1}{\sqrt{2}}\big)$ and $L=2\pi$, when $a+b=0$ and for $L>0$ and $k \in (0, \sfrac{1}{\sqrt{2}})$, when $a+b=\sfrac{1}{6}$.

Our paper is organized as follows: In Section \ref{section2} we present some basic notations presents in our paper. In Section \ref{Existenceofsolution}, we show the existence of periodic traveling wave solutions for the system \eqref{abcd}. In Section \ref{spectral-analysis}, we present spectral properties for the linearized operator $\mathcal{L}$ related to the ``abcd'' Boussinesq system \eqref{abcd}. Finally, our result about orbital stability associated to periodic traveling waves is shown in Section \ref{stability-section}.

\section{Notation}\label{section2}

Some basic notations concerning the periodic Sobolev spaces and other useful notations are introduced. For a more complete explanation of this topic, we refer the reader to see \cite{Iorio}. For a fixed $L>0$, $L^2_{per}:=L^2_{per}([0,L])$ is the space of all square (Lebesgue) integrable functions which are $L$-periodic. For $s\geq0$, the Sobolev space
$H^s_{per}:=H^s_{per}([0,L])$
is the set of all periodic distributions such that
$
\|f\|^2_{H^s_{per}}=L\sum_{n=-\infty}^{\infty}(1+|n|^2)^s|\hat{f}(n)|^2 <\infty,
$
where $\hat{f}$ is the periodic Fourier transform of $f$. The space $H^s_{per}$ is a  Hilbert space with natural inner product denoted by $(\cdot, \cdot)_{H^s}$. When $s=0$, the space $H^s_{per}$ is isometrically isomorphic to the space  $L^2_{per}$, that is, $L^2_{per}=H^0_{per}$. The norm and inner product in $L^2_{per}$ will be denoted by $\|\cdot \|_{L^2_{per}}$ and $(\cdot, \cdot)_{L^2_{per}}$. We use the notation $H_{per}^s\times H_{per}^s$ to refer to the cross product between spaces $H_{per}^s$ with usual norm and inner product.

The symbols $\sn(\cdot, k), \dn(\cdot, k)$ and $\cn(\cdot, k)$ represent the Jacobi elliptic functions of \textit{snoidal}, \textit{dnoidal}, and \textit{cnoidal} type, respectively. For $k \in (0, 1)$, $\K(k)$ and $\E(k)$ will denote the complete elliptic integrals of the first and second kind, respectively. For the precise definition and additional properties of the elliptic functions we refer the reader to  \cite{ByrdFriedman}.

We denote the number of negative eigenvalues and the dimension of the kernel of a certain linear operator $\mathcal{A}$, by $\text{n}(\mathcal{A})$ and $\text{z}(\mathcal{A})$, respectively.

\section{Existence of periodic waves}\label{Existenceofsolution}

In this section, our purpose is to explicit the existence of $L-$periodic solutions $(\varphi, \psi)$ associated to the system
\begin{equation}\label{SystemEDO3}
\begin{cases}
-\omega \varphi+\psi+\varphi \psi+a \psi''+b\omega \varphi''-A_1=0 \\
-\omega \psi+\varphi+\frac{1}{2}(\psi^2)+c\varphi''+d\omega \psi''-A_2=0
\end{cases}
\end{equation}
for some constants $A_1, A_2$. In whole this paper we are interested in solutions of the form $(\varphi, \psi)$ with $\psi = B \varphi$, where $B$ is a convenient constant. Motivated by \cite[Section 4]{ChenChenNguyen}, let us consider the ansatz
\begin{equation}\label{cnoidalsolution}
\varphi(x)=b_0+b_2 \, {\rm cn}^2\left( \frac{2{\rm K}(k)}{L} x, k \right),
\end{equation}
 where $b_0$ and $b_2$ are real constants to be chosen. Consider $a=c<0$ and $b=d>0$. Substituting $(\ref{cnoidalsolution})$ into the system \eqref{SystemEDO3}, we can obtain two basic cases:

\subsection{Case 1: $a+b = 0$.} We obtain in this case that parameters $b_0, b_2$ in $(\ref{cnoidalsolution})$ are given by
\begin{equation}\label{b0case1}
b_0=\frac{-64\,(B\omega - 1)^2 \,b\, \left(k^2 - \frac{1}{2}\right){\rm K}(k)^2 + 2\left( \left(\omega^2 + \frac{1}{2}\right)B - \frac{3}{2}{\omega}\right) B L^2}{2 \, B^2 \, L^2 (B\omega - 1)}
\end{equation}
and
\begin{equation}\label{b2case1}
b_2=\frac{48\,b \, k^2 \, (B\omega -1) {\rm K}(k)^2}{B^2 \, L^2},
\end{equation}
where $k \in \big(0, \sfrac{1}{\sqrt{2}}\big)$ and $\omega \in (-1,1)$. The constant  $B$ is then given by
\begin{equation}\label{Bcase1}
B=-\frac{\omega}{2}+\frac{1}{2}\sqrt{\omega^2+8}>0 \quad \text{or} \quad B=-\frac{\omega}{2}-\frac{1}{2}\sqrt{\omega^2+8}<0.
\end{equation}
In addition, the integration constants $A_1, A_2$ are expressed as
\begin{equation}\label{A1case1}
A_1=-{\frac { \left( B-\omega \right) ^{2} ( -256\,{b}^{2} \, r(k) \,  {\rm K}(k)^{4}+{L}^{4} )}{4\,B{L}^{4}}}
\end{equation}
and
\begin{equation}\label{A2case1}
A_2={\frac {1024\,{b}^{2} r(k)  \left( B\omega
-1 \right) ^{4}  {\rm K} (k)^{4}-
4\left( {B}^{3} s(\omega)+ \left( -3\,{\omega}
^{3}- \tfrac{3\omega}{2}\right) {B}^{2}+ \left( \tfrac{11}{4}{\omega}^{2}+1 \right) B-\omega
 \right)B{L}^{4}}{8\,{L}^{4}{B}^{2} \left( B\omega-1 \right)^{2}}},
\end{equation}
where $r(k)= {k}^{4}-{k}^{2}+1$ and $s(\omega):= {\omega}^{4}+{\omega}^{2}-\tfrac{1}{4}$.

Summarizing, we have the following result.

\begin{proposition}[Cnoidal waves solutions for the case $a+b = 0$]\label{cnoidalcurve1}
Let $L>0$ be fixed and consider $(b,\omega,k) \in (0,+\infty)\times (-1,1)\times \big(0,\sfrac{1}{\sqrt{2}}\big)$. If $a=c < 0$, $b=d > 0$ and $a+b = 0$, then the system \eqref{SystemEDO3} has $L$-periodic solutions $(\varphi, \psi)=(\varphi, B \varphi) \in C^\infty_{per}([0,L]) \times C^\infty_{per}([0,L])$ with $\varphi$ having the cnoidal profile in \eqref{cnoidalsolution}. Parameters $b_0, b_2, B$  depend smoothly on $b>0$, $\omega \in (-1,1)$ and $k \in \big(0,\sfrac{1}{\sqrt{2}}\big)$ and they are given by \eqref{b0case1}, \eqref{b2case1} and \eqref{Bcase1}, respectively. In addition, we have that $\varphi<0$.
\end{proposition}
\begin{flushright}
	$\blacksquare$
\end{flushright}
\begin{remark}
	The cnoidal solution $\varphi$ in $(\ref{solcn1})$, where $b_0$ and $b_2$ are given by $(\ref{b0case1})$ and $(\ref{b2case1})$, can be determined for all $k\in (0,1)$. The restriction to the case $k\in \big(0,\sfrac{1}{\sqrt{2}}\big)$ comes from the fact that we are considering only negative solutions in our study. The reason for that is because we need to obtain that $\mathcal{L}_2$ in \eqref{L1L2case12} is a positive operator to conclude, since $n(\mathcal{L}_1)=1$, that $n(\mathcal{L})=1$ (see Lemmas $\ref{lemma-L2}$ and $\ref{lemma-L1}$ in Section $\ref{spectral-analysis}$).
\end{remark}
\subsection{Case 2: $a + b = \tfrac{1}{6}$.} In a general case, without considering $a=c$ and $b=d$, parameters $b_0, b_2$ are given by
\begin{equation}\label{b0case2}
b_0=\frac{ -64 \, (B\,d \, \omega +c)^2 \left(k^2 - \frac{1}{2} \right) {\rm K}(k)^2 - L^2 \left( (-2\,d \,\omega^2 +a)\,B + \omega(b-2c)\right) B}{2\, B^2 \, L^2 \, (B \,d \, \omega + c)}
\end{equation}
and
\begin{equation}\label{b2case2}
b_2=\frac{48\, k^2 \, (B \, d \, \omega + c) {\rm K}(k)^2}{B^2 \, L^2}.
\end{equation}

Constants of integration $A_1, A_2$ can be expressed by 
\begin{equation}\label{A1case2}
A_1 = \frac{ 256 \left( k^4 - k^2 + 1 \right) (Ba + b\omega)^2 {\rm K}(k)^4 - L^4(B-\omega)^2}{4\, B \, L^4}
\end{equation}
and
\begin{equation}\label{A2case2}
	A_2 = - \frac{-1024 r(k) (Bd\omega + c)^4 {\rm K}(k)^4 - \left( \beta_1 B^3 + 2 \beta_2 \omega B^2 + \beta_3 B + 4c^2 \omega\right) L^4 B}{8\, B^2 \, L^4 \, (Bd\omega +c)^2},
\end{equation}
where
\begin{equation*}
	\beta_1 = (-4d^2 \omega^4 - 4d^2 \omega^2 + a^2), \ \beta_2 =(-4cd + 2d^2) \omega^2 + ab - 4cd, \
	\beta_3 = (b^2 - 4c^2 + 8cd) \omega^2 - 4c^2.
\end{equation*}

The wave speed $\omega \in  \mathbb{R}$ is given by
\begin{equation}\label{omegacase2}
\omega=\pm{\frac {2(a-c)}{\sqrt {-2\,ab+2\,ac+4\,ad+2\,{b}^{2}-2\,bc-8\,bd+4\,c
d+8\,{d}^{2}}}}.
\end{equation}
Parameter $B$ is  
\begin{equation}\label{Bcase2}
B=\pm {\frac {\sqrt {2}\sqrt {-b+c+2d}}{\sqrt {a-b+2\,d}}}.
\end{equation}

%\begin{remark}
%We note that by not imposing any conditions on the constants $a,b,c,d \in \mathbb{R}$ when considering ansatz \eqref{cnoidalsolution}, we get $\omega \in \mathbb{R}$ being a fixed constant, differently on the Case $1$, where we obtain $\omega \in (-1,1) $ being a free parameter. Furthermore, we can see that we cannot consider the constants $a,b,c,d \in \mathbb{R}$ as in Case $1$.
%\end{remark}

Thus, by considering $a=c < 0$ and $b=d > 0$, where $a+b = \tfrac{1}{6}$, we obtain 
\begin{equation}\label{b0-case2-w=0}
	b_0 = - \frac{1}{2} \frac{(32 k^2 - 16)\, a \, {\rm K}(k)^2 + L^2}{L^2}
\end{equation}
and 
\begin{equation}\label{b2-case2-w=0}
	b_2 = \frac{24 {\rm K}(k)^2 \, a \, k^2}{L^2}.
\end{equation}

It also follows that $B = \pm \sqrt{2}$ and $\omega = 0$. In addition, constants $A_1$ and $A_2$ are independent of the choice of $B=\pm\sqrt{2}$ and they are given by
\begin{equation}\label{A1-case2-w=0}
	A_1 = \frac{1}{4} \frac{ \sqrt{2} \, (-256 \, {\rm K}(k)^4 \, a^2 \, r(k) + L^4)}{L^4}
\end{equation}
and
\begin{equation}\label{A2-case2-w=0}
	A_2 = -\frac{1}{4} \frac{ (-256 \, {\rm K}(k)^4 \, a^2 \, r(k) + L^4)}{L^4}.
\end{equation}

Summarizing, we have the following result.

\begin{proposition}[Cnoidal waves solutions for the case $a+b = \sfrac{1}{6}$]\label{cnoidalcurve2}
Let $L>0$ be fixed and consider $(a,k) \in \left(-\infty,0\right)\times (0,1)$. If $a=c<0$, $b=d>0$ and $a+b = \sfrac{1}{6}$, then the equation \eqref{SystemEDO3} has $L$-periodic solutions $(\varphi, \psi)=(\varphi, B \varphi) \in C^\infty_{per}([0,L]) \times C^\infty_{per}([0,L])$  with $\varphi$ having the cnoidal profile in \eqref{cnoidalsolution}. Parameters $b_0$ and $b_2$  depend smoothly on $a\in \left(-\infty,0\right)$ and  $k \in (0,1)$ and they are given by \eqref{b0-case2-w=0} and \eqref{b2-case2-w=0}, respectively. In addition, we have that $\omega=0$ and $B=\pm\sqrt{2}$.% and $\varphi<0$. 

\end{proposition}
\begin{flushright}
	$\blacksquare$
\end{flushright}

\section{Spectral analysis}\label{spectral-analysis}

\subsection{Floquet Theory Framework} \label{Floquet}
Before presenting the spectral analysis concerning the operators in \eqref{matrixop}, we need to recall some basic facts about the Floquet theory (for further details see \cite{Eastham} and \cite{MagnusWinkler}).

For $n\geq1$ integer, let $\mathcal{V}\subset\mathbb{R}^n$ be an open set. For a given $r\in\mathcal{V}$, consider $\phi$ a solution of the general equation
\begin{equation}\label{geneq}
	-\phi''(x)+g(r,\phi(x))=0,
	\end{equation}
where $g$ is smooth in all variables and $\phi$ is $L_r$-periodic. Let 
$$
\mathcal{P}_{r}:H_{per}^2([0,L_r]) \subset L_{per}^2([0,L_r]) \rightarrow L_{per}^2([0,L_r])
$$ be the associated Hill operator given by
\begin{equation}\label{operatorP}
\mathcal{P}_r=-\partial_x^2+g'(r, \phi (x)),
\end{equation}
where $g$ is smooth in all variables and $g'$ indicates the derivative with respect to $\phi$.\\
\indent According to \cite[Theorem 2.1]{MagnusWinkler}, the spectrum of $\mathcal{P}_r$  is formed by an unbounded sequence of
real eigenvalues $(\lambda_n)_{n \in \mathbb{N}} \subset \mathbb{R}$ so that
\[
\lambda_0 < \lambda_1 \leq \lambda_2 <\lambda_3 \leq \lambda_4 <
\cdots\; < \lambda_{2n-1} \leq \lambda_{2n}\; \cdots,
\]
where equality means that $\lambda_{2n-1} = \lambda_{2n}$  is a double eigenvalue. Moreover, the spectrum  is characterized by the number of zeros of the eigenfunctions as: if $p$ is an eigenfunction associated to either $\lambda_{2n-1}$ or $\lambda_{2n}$, then $p$  has exactly $2n$ zeros in the half-open interval $[0, L_r)$.% In particular, the even eigenfunction associated to the first eigenvalue $\lambda_0$ has no zeros in $[0, L]$.

Consider the general Hill equation
\begin{equation}\label{zeqL}
-f''+g'(r, \phi (x))f=0,
\end{equation}
where $f \in C_b^{\infty}([0,L_r])$, with  $C_b^{\infty}([0,L_r])$  indicating the space constituted by bounded real-valued smooth functions. If $p$ is a periodic solution of $(\ref{zeqL})$, we obtain by the classical Floquet theory in \cite{Eastham} and \cite{MagnusWinkler} the existence of a solution $y$ for the equation \eqref{zeqL} which is linearly independent with $p$ such that $\{p,y\}$ forms a fundamental set of solutions for the  Hill equation \eqref{zeqL}. Moreover,  there exists  $\Theta \in   \mathbb{R}$ (depending on $y$ and $p$) such that
\begin{equation}\label{theta0}
y(x+L)=y(x)+\Theta p(x) \quad \mbox{for all} \quad x\in\mathbb{R}.
\end{equation}
Constant $\Theta$ measures how function $y$ is periodic. To be more precise, $\Theta=0$ if and only if $y$ is periodic. This criterion is very useful to establish if the kernel of $\mathcal{P}_r$ is $1$-dimensional by proving that $\Theta\neq0$. Concerning this fact, we have the following result.

\begin{proposition}\label{theta}
If $\Theta$ is the constant given in \eqref{theta0}, then the eigenvalue $\lambda=0$ is simples if and only if $\Theta \neq 0$. Moreover, if $\Theta \neq 0$, then $\lambda_1=0$ if $\Theta<0$, and $\lambda_2=0$ if $\Theta>0$.
\end{proposition}
\begin{proof}
See \cite[Theorem 3.3]{NataliNeves2014}.
\end{proof}

We also need the concept of isoinertial family of self-adjoint operators.

\begin{definition}\label{defi12}
Consider $r \in \mathcal{V}$. The inertial index of the operator $\mathcal{P}_{r}$ is a pair ${\rm in}(\mathcal{P}_{r})=(n,z)\in \mathbb{N}^2$, where $n \in \mathbb{N}$ denotes the dimension of the negative subspace of  and $z \in \mathbb{N}$ denotes the dimension of  $\Ker(\mathcal{P}_{r})$.
\end{definition}

\begin{definition}\label{defi1234}
The family of linear operators  $\{\mathcal{P}_{r}\;  ; \; r \in \mathcal{V}\}$ is said to be isoinertial if
 ${\rm in}(\mathcal{P}_{r})$ is constant for any $r \in \mathcal{V}$.
\end{definition}

We have the following result concerning the behaviour of the non-positive spectrum of the linear operator  $\mathcal{P}_{r}$ defined in \eqref{operatorP} by knowing it for a fixed value $r_0 \in \mathcal{V}$.

\begin{proposition}\label{isonertialP}
Let $r \in \mathcal{V}$ and $\mathcal{P}_{r}$ be the Hill Operator defined in \eqref{operatorP}. If $\lambda=0$ is an eigenvalue of  $\mathcal{P}_{r}$ and $g$ is of class $C^1$, then the family of operators $\{\mathcal{P}_{r}\;  ; \; r \in \mathcal{V}\}$ is isoinertial with respect to the parameters $r \in \mathcal{V}$. In particular, if $\Theta<0$ (respectively $\Theta>0$) for some $r_0 \in \mathcal{V}$, then $\Theta<0$ (respectively $\Theta>0$) for all $r \in \mathcal{V}$. Moreover, the inertial index ${\rm in}(\mathcal{P}_r)$ is constant in terms of the period $L_r$ of the function $\varphi$. In particular, if $\Theta<0$ (respectively $\Theta>0)$ for some $L_{r_0}>0$ in a convenient open interval, then $\Theta<0$ (respectively $\Theta>0$) for all $L_r>0$ in the same interval.
\end{proposition}
\begin{proof}
See \cite[Theorems 3.8 and 3.12]{NataliPastorCristofani}.
\end{proof}

\begin{remark}\label{rem12} Theorem 3.12 in \cite{NataliPastorCristofani} determines that the inertial index ${\rm in}(\mathcal{P}_r)$ is also constant in terms of the period function $L_r$, $r\in\mathcal{V}\subset\mathbb{R}^n$. Our periodic solutions $\varphi$ in $(\ref{cnoidalsolution})$ are considered with fixed period and depending on $(b,\omega,k)\in (0,+\infty)\times(-1,1)\times \big(0,\sfrac{1}{\sqrt{2}}\big)$ when $a+b=0$ or even depending on $(a,k)\in (-\infty,0 )\times(0,1)$ when $a+b=\sfrac{1}{6}$. Since both periodic waves $\varphi$ exist for an arbitrary but fixed period $L>0$, we can calculate the inertia index for a determined period $L_0>0$. By using Theorem 3.12 in \cite{NataliPastorCristofani}, we conclude that the inertial index will be the same for $L>0$ arbitrary but fixed and not depending on the parameters $b$, $\omega$ and $k$.	
\end{remark}

\subsection{Spectral Analysis of the Linearized Operator}

We are going to use the Floquet Theory and its improvement in \cite[Section 3]{NataliPastorCristofani} (see also \cite{NataliNeves2014}) to obtain the spectral properties required in assumption (H1) for the linearized operator $\mathcal{L}$ in $(\ref{matrixop})$ around the periodic wave $(\varphi, \psi)=(\varphi, B \varphi)$.

\indent Let $L>0$ be fixed. Consider  $(\varphi, \psi)=(\varphi, B \varphi)$ the solution obtained in Propositions \ref{cnoidalcurve1} and \ref{cnoidalcurve2}. We study the spectral properties of the matrix operator 
$$\mathcal{L}: H^2_{per} \times H^2_{per} \subset L^2_{per} \times L^2_{per} \longrightarrow L^2_{per} \times L^2_{per} $$ given by
\begin{equation}\label{operatorL}
 \mathcal{L}=\begin{pmatrix}
1+c\partial_x^2 & b\omega \partial_x^2+\psi-\omega \\
b \omega \partial_x^2+\psi-\omega &1+a\partial_x^2+\varphi
\end{pmatrix}
\end{equation}
for $a,b,c,d$ satisfying $a=c <0$ and $b=d > 0$. First, we observe that $\mathcal{L}$ is a self-adjoint operador and, by \eqref{SystemEDO1}, it is clear that $\mathcal{L}(\varphi', \psi')=(0,0)$. Since solutions $\varphi$ and $\psi$ are multiple each other, we can proceed as in \cite{HakkaevStanislavovaStefanov} to get the decomposition
\begin{equation}\label{similartransformation1}
\mathcal{L}=T^{-1} M T
\end{equation}
where
\begin{equation}\label{matrixT}
T=\begin{pmatrix}
\tfrac{1}{\sqrt{2}} & \tfrac{1}{\sqrt{2}} \\ 
-\tfrac{1}{\sqrt{2}} & \tfrac{1}{\sqrt{2}}
\end{pmatrix} 
\end{equation}
and
\begin{equation}\label{matrixM}
M=\begin{pmatrix}
-\partial_x^2(-a-b\omega)+(1-\omega)+\psi+\tfrac{\varphi}{2} & \tfrac{\varphi}{2} \\ 
\tfrac{\varphi}{2} & -\partial_x^2(-a+b\omega)+(1+\omega)-\psi+\tfrac{\varphi}{2}
\end{pmatrix}.
\end{equation}

To perform the spectral analysis of the operator $\mathcal{L}$ we will split our spectral analysis into two cases.

\subsection{Spectral Analysis for the Case 1: $a+b = 0$.}\label{EspectralAnalysisCase1} We can decompose matrix $M$ in \eqref{matrixM} to write $\mathcal{L}$ as a diagonal operator of the form
\begin{equation}\label{opL-matriz-similar}
	\mathcal{L} = (UST)^{-1} \left( 
	\begin{array}{cc}
		\mathcal{L}_1 & 0 \\
		0 & \mathcal{L}_2
	\end{array}
\right) (UST),
\end{equation}
where $T$ is given in \eqref{matrixT},  %$U$ is an orthogonal matrix  and $S$ is given by
\begin{equation*}\label{matrixS}
	S = \left( \begin{array}{cc}
		\sqrt{b(1-\omega)} & 0 \\
		0 & \sqrt{b(1+\omega) }
		\end{array}
	\right)\ 
	\mbox{and}\
U=\begin{pmatrix}
\alpha
   & \beta
 \\ 
\noalign{\medskip}{\frac {\alpha\, \left( 2\,B+\omega-3 \right) }{
\sqrt {1-{\omega}^{2}}}}
 &  -{\frac {\beta\, \left( 2\,B+\omega+3 \right) }{
\sqrt {1-{\omega}^{2}}}}
\end{pmatrix},
\end{equation*} 
with
\begin{equation*}
\alpha = -\sqrt {b \left( 1+\omega \right) } \left(B-1\right) \left( 2\,B+\omega+3 \right),
\end{equation*}
and
\begin{equation*}
\beta = -\sqrt {b \left( 1-\omega \right) } \left( 1+B \right)  \left( -2\,B-\omega+
3 \right).
\end{equation*}
The Hill operators $\mathcal{L}_1, \mathcal{L}_2: H^2_{per} \subset L^2_{per} \rightarrow L^2_{per}$ in \eqref{opL-matriz-similar} are defined as 
\begin{equation}\label{L1L2case1}
	\mathcal{L}_1 = -\partial_x^2 + \frac{1}{b} + \ell_1 \varphi \quad \text{ and } \quad \mathcal{L}_2 = -\partial_x^2 + \frac{1}{b} + \ell_2 \varphi,
\end{equation}
where
\begin{equation}\label{l1l2}
	\ell_1 = \frac{4 + 2\, B\,\omega}{2\, b \, (1-\omega^2)} \quad \text{ and } \quad \ell_2 = \frac{-2 + 2\, B\,\omega }{2\, b \, (1-\omega^2)}.
\end{equation}

The operators $\mathcal{L}_1$ and $\mathcal{L}_2$ in \eqref{L1L2case2} play an important role in our spectral analysis. In fact, using the Sylvester law of inertia we obtain
$
\sigma(\mathcal{L})=   \sigma(\mathcal{L}_1) \cup  \sigma(\mathcal{L}_2).   
$
In particular, one has
\begin{equation}\label{relationnegative}
{\rm n}(\mathcal{L})= {\rm n}(\mathcal{L}_1)+{\rm n}(\mathcal{L}_2) \quad \text{and} \quad    {\rm z}(\mathcal{L})= {\rm z}(\mathcal{L}_1)+{\rm z}(\mathcal{L}_2).
\end{equation}

Since $B$ is given in \eqref{Bcase1}, we obtain that $\ell_2 < 0$ and $\varphi<0$ for all $b>0$, $k \in \big( 0, \sfrac{1}{\sqrt{2}} \big)$ and $\omega \in (-1,1)$.  Thus, $\mathcal{L}_2$ is a positive operator, the number of negative eigenvalues of $\mathcal{L}_2$ is zero and ${\rm Ker}(\mathcal{L}_2) = \{0\}$. 

Summarizing the above, we can establish the following result.
\begin{lemma}\label{lemma-L2}
	Let $(\varphi, \psi) = (\varphi_{b,\omega,k}, B \varphi_{b,\omega,k})$ be the pair of periodic solution given by Proposition $\ref{cnoidalcurve1}$. The spectrum of $\mathcal{L}_2$ in \eqref{L1L2case1} is only constituted by a discrete set of positive eigenvalues.
\end{lemma}
\begin{flushright}
	$\blacksquare$
\end{flushright}

\indent Next, we analyze the operator $\mathcal{L}_1$. In fact, after a tedious calculation, we can see that  $\mathcal{L}_1\varphi' = 0$, that is, $\lambda = 0$ is an eigenvalue of $\mathcal{L}_1$ whose associated eigenfunction is $\varphi'$. Since $\varphi'$ has two zeros in $[0,L)$, we obtain by the Floquet theory that $\lambda=0$  is the second or the third eigenvalue of $\mathcal{L}_1$. Our aim is to show that actually $\lambda=0$ is the second eigenvalue of $\mathcal{L}_1$. In fact, as far as we can see, there exists a function $y \in C_b^\infty([0,L])$ satisfying the Hill equation
\begin{equation}\label{Hillequation1}
	-y'' + \frac{1}{b} y + \ell_1 \varphi y = 0,
\end{equation}
where $\{\varphi',y\}$ is the fundamental set of solutions for \eqref{Hillequation1} and $y$ can be periodic or not. Since $\varphi'$ is an odd function, we obtain that $y$ is even and satisfying the following initial value problem
\begin{equation}\label{PVI1}
	\left\{ 
	\begin{array}{l}
		-y'' + \frac{1}{b} y + \ell_1 \varphi y = 0, \\
	 y(0) = -\frac{1}{\varphi''(0)}, \\
		 y'(0) = 0.
\end{array}
\right.
\end{equation}
The constant $\Theta$ appearing in $(\ref{theta0})$ is then given by 
\begin{equation}\label{theta1}
	\Theta = \frac{y'(L)}{\varphi''(0)}.
\end{equation}

First, we consider $B>0$ in \eqref{Bcase1}. To evaluate the sign of the constant $\Theta$ in $(\ref{theta1})$, we use Proposition $\ref{isonertialP}$ and Remark $\ref{rem12}$ to obtain, for a fixed $L>0$, the isoinertially of the family of operators $\mathcal{L}_1$ in terms of the parameters $ b > 0,\ \omega \in (-1,1)$ and $k \in (0,\sfrac{1}{\sqrt{2}})$. In order to improve the comprehension of the reader, we use the \textit{software Mathematica} to show some tables with the behaviour of the constant $\Theta$ by considering these change of parameters. 
\begin{table}[!htb]
	\begin{tabular}{|c|c|}
		\hline 
		\multicolumn{2}{|c|}{$k=0.5 : \omega =0$} \\ 
		\hline 
		$b$ & $\Theta$ \\
		\hline 
		$1$ & $-1.37\times 10^{-5} $ \\
		\hline
		$2$ & $-3.43\times 10^{-6}$ \\
		\hline 
		$3$ & $-1.52 \times 10^{-6}$ \\
		\hline 
		$4$ & $-8.57\times 10^{-7}$ \\
		\hline 
		$5$ & $-5.49 \times 10^{-7}$ \\
		\hline 
		$10$ & $-1.37 \times 10^{-7}$ \\
		\hline 
		$20$ & $-3.43\times 10^{-8}$ \\
		\hline 
	\end{tabular}
	\hspace{0.5cm} 
\begin{tabular}{|c|c|}
	\hline 
	\multicolumn{2}{|c|}{$k=0.3 : \omega =0.5$} \\ 
	\hline 
	$b$ & $\Theta$ \\
	\hline 
	$1$ & $-4.05\times 10^{-5}$ \\
	\hline
	$2$ & $-1.01\times 10^{-5}$ \\
	\hline 
	$3$ & $-4.50\times 10^{-6}$ \\
	\hline 
	$4$ & $-2.53\times 10^{-6}$ \\
	\hline 
	$5$ & $-1.62\times 10^{-6}$ \\
	\hline 
	$10$ & $-4.05\times 10^{-7}$ \\
	\hline 
	$20$ & $-1.01\times 10^{-7}$ \\
	\hline 
\end{tabular}
	\hspace{0.5cm} 
	\begin{tabular}{|c|c|}
		\hline 
		\multicolumn{2}{|c|}{$k=0.7 : \omega =-0.5$} \\ 
		\hline 
		$b$ & $\Theta$ \\
		\hline 
		$1$ & $-8.67\times 10^{-6}$ \\
		\hline
		$2$ & $-2.16\times 10^{-6}$ \\
		\hline 
		$3$ & $-9.63\times 10^{-7}$ \\
		\hline 
		$4$ & $-5.42\times 10^{-7}$ \\
		\hline 
		$5$ & $-3.46 \times 10^{-7}$ \\
		\hline 
		$10$ & $-8.67\times10^{-8}$ \\
		\hline 
		$20$ & $-2.16\times 10^{-8}$ \\
		\hline 
	\end{tabular}
	\hspace{0.5cm} 
	\begin{tabular}{|c|c|}
		\hline 
		\multicolumn{2}{|c|}{$k=0.1 : \omega =0.9$} \\ 
		\hline 
		$b$ & $\Theta$ \\
		\hline 
		$1$ & $-8.05 \times 10^{-4}  $ \\
		\hline
		$2$ & $-2.01  \times 10^{-4}$ \\
		\hline 
		$3$ & $-8.95  \times 10^{-5}$ \\
		\hline 
		$4$ & $-5.03  \times 10^{-5}$ \\
		\hline 
		$5$ & $-3.22  \times 10^{-5}$ \\
		\hline 
		$10$ & $-8.06\times 10^{-6}$ \\
		\hline 
		$20$ & $-2.01\times 10^{-6}$ \\
		\hline 
	\end{tabular}
\vspace{0.3cm}
	\caption{Values of $\Theta$ in $(\ref{theta1})$ for $L=1$ and $B>0$.}
	\label{tableL=1}
\end{table}
\begin{table}[!h]
	\begin{tabular}{|c|c|}
		\hline 
		\multicolumn{2}{|c|}{$k=0.5 : \omega =0$} \\ 
		\hline 
		$b$ & $\Theta$ \\
		\hline 
		$1$ & $-5.3074$ \\
		\hline
		$2$ & $-1.3268$ \\
		\hline 
		$3$ & $-0.5897$ \\
		\hline 
		$4$ & $-0.3317$ \\
		\hline 
		$5$ & $-0.2122$ \\
		\hline 
		$10$ & $-0.0530$ \\
		\hline 
		$20$ & $-0.0132$ \\
		\hline 
	\end{tabular}
	\hspace{0.5cm} 
\begin{tabular}{|c|c|}
	\hline 
	\multicolumn{2}{|c|}{$k=0.3 : \omega =0.5$} \\ 
	\hline 
	$b$ & $\Theta$ \\
	\hline 
	$1$ & $-15.6703$ \\
	\hline
	$2$ & $-3.9175$ \\
	\hline 
	$3$ & $-1.7411$ \\
	\hline 
	$4$ & $-0.9793$ \\
	\hline 
	$5$ & $-0.6268$ \\
	\hline 
	$10$ & $-0.1567$ \\
	\hline 
	$20$ & $-0.0391$ \\
	\hline 
\end{tabular}
	\hspace{0.5cm} 
	\begin{tabular}{|c|c|}
		\hline 
		\multicolumn{2}{|c|}{$k=0.7 : \omega =-0.5$} \\ 
		\hline 
		$b$ & $\Theta$ \\
		\hline 
		$1$ & $-3.3537$ \\
		\hline
		$2$ & $-0.8384$ \\
		\hline 
		$3$ & $-0.3726$ \\
		\hline 
		$4$ & $-0.2096$ \\
		\hline 
		$5$ & $-0.1341$ \\
		\hline 
		$10$ & $-0.0335$ \\
		\hline 
		$20$ & $-0.0083$ \\
		\hline 
	\end{tabular}
	\hspace{0.5cm} 
	\begin{tabular}{|c|c|}
		\hline 
		\multicolumn{2}{|c|}{$k=0.1 : \omega =0.9$} \\ 
		\hline 
		$b$ & $\Theta$ \\
		\hline 
		$1$ & $-311.4599$ \\
		\hline
		$2$ & $-77.8617$ \\
		\hline 
		$3$ & $-34.6085$ \\
		\hline 
		$4$ & $-19.4654$ \\
		\hline 
		$5$ & $-12.4579$ \\
		\hline 
		$10$ & $-3.1145$ \\
		\hline 
		$20$ & $-0.7786$ \\
		\hline 
	\end{tabular}
\vspace{0.3cm}
	\caption{Values of $\Theta$ in $(\ref{theta1})$ for $L=2\pi$ and $B>0$.}
	\label{tableL=2pi}
\end{table}

Next, we consider $B<0$ in \eqref{Bcase1} for some values of $b>0$, $k \in \big(0, \sfrac{1}{\sqrt{2}}\big)$ and $\omega \in (-1,1)$. We also have $\Theta<0$, as expected.
\begin{table}[!h]
	\begin{tabular}{|c|c|}
		\hline 
		\multicolumn{2}{|c|}{$k=0.5 : \omega =0$} \\ 
		\hline 
		$b$ & $\Theta$ \\
		\hline 
		$1$ & $-5.3074$ \\
		\hline
		$2$ & $-1.3268$ \\
		\hline 
		$3$ & $-0.5897$ \\
		\hline 
		$4$ & $-0.3317$ \\
		\hline 
		$5$ & $-0.2122$ \\
		\hline 
		$10$ & $-0.0530$ \\
		\hline 
		$20$ & $-0.0132$ \\
		\hline 
	\end{tabular}
	\hspace{0.5cm} 
\begin{tabular}{|c|c|}
	\hline 
	\multicolumn{2}{|c|}{$k=0.3 : \omega =0.5$} \\ 
	\hline 
	$b$ & $\Theta$ \\
	\hline 
	$1$ & $-15.6703$ \\
	\hline
	$2$ & $-3.9175$ \\
	\hline 
	$3$ & $-1.7411$ \\
	\hline 
	$4$ & $-0.9793$ \\
	\hline 
	$5$ & $-0.6268$ \\
	\hline 
	$10$ & $-0.1567$ \\
	\hline 
	$20$ & $-0.0391$ \\
	\hline 
\end{tabular}
	\hspace{0.5cm} 
	\begin{tabular}{|c|c|}
		\hline 
		\multicolumn{2}{|c|}{$k=0.7 : \omega =-0.5$} \\ 
		\hline 
		$b$ & $\Theta$ \\
		\hline 
		$1$ & $-3.3537$ \\
		\hline
		$2$ & $-0.8384$ \\
		\hline 
		$3$ & $-0.3726$ \\
		\hline 
		$4$ & $-0.2096$ \\
		\hline 
		$5$ & $-0.1341$ \\
		\hline 
		$10$ & $-0.0335$ \\
		\hline 
		$20$ & $-0.0083$ \\
		\hline 
	\end{tabular}
	\hspace{0.5cm} 
	\begin{tabular}{|c|c|}
		\hline 
		\multicolumn{2}{|c|}{$k=0.1 : \omega =0.9$} \\ 
		\hline 
		$b$ & $\Theta$ \\
		\hline 
		$1$ & $-311.4599$ \\
		\hline
		$2$ & $-77.8617$ \\
		\hline 
		$3$ & $-34.6085$ \\
		\hline 
		$4$ & $-19.4654$ \\
		\hline 
		$5$ & $-12.4579$ \\
		\hline 
		$10$ & $-3.1145$ \\
		\hline 
		$20$ & $-0.7786$ \\
		\hline 
	\end{tabular}
\vspace{0.3cm}
	\caption{Values of $\Theta$ in $(\ref{theta1})$ for $L=2\pi$ and $B<0$.}
	\label{tableL=2piII}
	\end{table}

Summarizing the arguments above, we can establish the following result.
\begin{lemma}\label{lemma-L1}
Let $(\varphi, \psi) = (\varphi_{b,\omega,k}, B \varphi_{b,\omega,k})$ be the pair of periodic solutions given by Proposition $\ref{cnoidalcurve1}$. The operator $\mathcal{L}_1$ defined in \eqref{L1L2case1} has exactly one negative eigenvalue which is simple and zero is the second eigenvalue which is simple with associated eigenfunction $\varphi'$. Moreover, the remainder of the spectrum is constituted by a discrete set of eigenvalues.
\end{lemma}
\begin{flushright}
	$\blacksquare$
\end{flushright}
\indent By Lemmas \ref{lemma-L1} and \ref{lemma-L2}, we have the following result.

\begin{proposition}\label{simplekernel1}
Let $(\varphi, \psi) = (\varphi_{b,\omega,k}, B \varphi_{b,\omega,k})$ be the pair of periodic solutions given by Proposition $\ref{cnoidalcurve1}$. The operator $\mathcal{L}$ defined in \eqref{operatorL} has exactly one negative simple eigenvalue and zero is a simple eigenvalue with associated eigenfunction $(\varphi', B \varphi')$. Moreover, the remainder of the spectrum is constituted by a discrete set of eigenvalues.
\end{proposition}

\begin{flushright}
	$\blacksquare$
\end{flushright}

\subsection{Spectral Analysis for the Case 2: $a+b = \tfrac{1}{6}$.}\label{EspectralAnalysisCase2} Recall in this case that $B=\pm \sqrt{2}$. We consider first $B=\sqrt{2}$ to write
\begin{equation}\label{operatorLcase2}
\mathcal{L}=(\mathcal{S}T)^* \begin{pmatrix}
\mathcal{L}_3 & 0 \\ 
0 &  \mathcal{L}_4
\end{pmatrix}(\mathcal{S}T)
\end{equation}
where $T$ is the matrix given by \eqref{matrixT}, $\mathcal{S}$ is defined as
\begin{equation}\label{matrixScase2}
\mathcal{S}= \tfrac{1}{\sqrt{6}} \begin{pmatrix}
\sqrt{3+2\sqrt{2}} & \sqrt{3-2\sqrt{2}} \\ 
-\sqrt{3-2\sqrt{2}} & \sqrt{3+2\sqrt{2}}
\end{pmatrix} 
\end{equation}
and $\mathcal{L}_3, \mathcal{L}_4: H^2_{per} \subset L^2_{per} \rightarrow L^2_{per}$ are the Hill operators 
\begin{equation}\label{L1L2case2}
	\mathcal{L}_3=-a\partial_x^2+1+2\varphi \quad \text{and} \quad \mathcal{L}_4=-a\partial_x^2+1-\varphi.
\end{equation}

\indent For the case $B=-\sqrt{2}$, we have the same representation of $\mathcal{L}$ as in \eqref{operatorLcase2} with
\begin{equation}\label{matrixScase2-2}
\mathcal{S}= \tfrac{1}{\sqrt{6}} \begin{pmatrix}
\sqrt{3-2\sqrt{2}} & \sqrt{3+2\sqrt{2}} \\ 
\sqrt{3+2\sqrt{2}} & -\sqrt{3-2\sqrt{2}}
\end{pmatrix}. 
\end{equation}
In both cases for $B$, it is easy to see that $\mathcal{S}T$ is an orthogonal matrix, so that 
\begin{equation}\label{STorthogonal}
	(\mathcal{S}T)^*=(\mathcal{S}T)^{-1}.
\end{equation}

Operators $\mathcal{L}_3$ and $\mathcal{L}_4$ in \eqref{L1L2case2} play an important role in the spectral analysis since we have 
\begin{equation}\label{relationnegative1}
	{\rm n}(\mathcal{L})= {\rm n}(\mathcal{L}_3)+{\rm n}(\mathcal{L}_4) \quad \text{and} \quad    {\rm z}(\mathcal{L})= {\rm z}(\mathcal{L}_3)+{\rm z}(\mathcal{L}_4).
\end{equation}

To establish our orbital stability result in the next section and since we do not have a smooth curve of periodic waves depending on the wave speed $\omega$ as in the case $a+b=0$, we need to do some useful simplifications in our approach. In fact, by assuming that $A_2 =0$ in \eqref{SystemEDO3}, we automatically have $A_1 = 0$. In this case, $a < 0$ is given as function of the modulus $k\in (0,1)$  %$k\in\left(0,\frac{1}{\sqrt{2}}\right)$ 
\begin{equation}\label{relationa}
a=- \frac{1}{16}\,{\frac {{L}^{2}}{ {\rm K} \left( k \right)^{2}\sqrt {{k}^{4}-{k}^{2}+1}}}<0,
\end{equation}
so that the smooth surface in Proposition $\ref{cnoidalcurve2}$ is a smooth curve depending on the modulus $k$. Next, by \eqref{relationa} we also have
$
b=\frac{1}{6}+ \frac{1}{16}\,{\frac {{L}^{2}}{ {\rm K} \left( k \right)^{2}\sqrt {{k}^{4}-{k}^{2}+1}}}>0.
$

\begin{remark}\label{relationsab}
Notice that by considering $a+b = \tfrac{1}{6}$, we obtain the surface tension $\tau = 0$. In this case, parameters $a,b,c,d$ satisfy the equalities given in \eqref{relationabcd1} for $\theta^2 = \tfrac{2}{3} $. As consequence of the first identity in \eqref{relationabcd2}, we have $a=\tfrac{\nu}{6}$ and
\begin{equation}\label{relation-nu}
\nu= \nu_k=-\frac{3}{8}\frac{L^2}{{\rm K}(k)^2 \,\sqrt{k^4-k^2+1}}.
\end{equation}
In addition, we obtain in this case that $\varphi<0$.
\end{remark}

From Remark $\ref{relationsab}$, we have obtained that $\varphi < 0$ for all $k \in (0,1)$, so that $\mathcal{L}_4$ is a positive operator. This fact implies that ${\rm n}(\mathcal{L}_4)=0$ and ${\rm Ker}(\mathcal{L}_4) = \{0\}$. 

\begin{lemma}\label{lemma-L4}
	Let $(\varphi, \psi) = (\varphi_{k}, B \varphi_{k})$ be the pair of periodic solution given by Proposition $\ref{cnoidalcurve2}$ with $A_1 = A_2 = 0$. The spectrum of $\mathcal{L}_4$ in \eqref{L1L2case1} is only constituted by a discrete set of positive eigenvalues.
\end{lemma}
\begin{flushright}
	$\blacksquare$
\end{flushright}

\indent Next, we analyze the spectrum of the operator $\mathcal{L}_3$. Using the first equation of \eqref{SystemEDO1} and the fact that $\omega=0$, we see that $\mathcal{L}_3\varphi' = 0$ and $\lambda = 0$ is an eigenvalue of $\mathcal{L}_3$. Since $\varphi'$ has two zeros in $[0,L)$, we deduce from the Floquet theory that $\lambda=0$  is the second or the third eigenvalue of $\mathcal{L}_3$. We show actually $\lambda=0$ is the second eigenvalue of $\mathcal{L}_3$, so that $\mathcal{L}_3$ has only one negative eigenvalue which is clearly simple. Let $y \in C_b^\infty([0,L])$ be a solution satisfying the Hill equation
\begin{equation}\label{Hillequation2}
	-y'' + y + 2 \varphi y = 0,
\end{equation}
with $\{\varphi',y\}$ being the fundamental set of solutions for \eqref{Hillequation2}. Since $\varphi'$ is an odd function, one has that $y$ is even and it satisfies the following initial value problem
\begin{equation}\label{PVI2}
	\left\{ 
	\begin{array}{l}
		-y'' +  y +2\varphi y = 0, \\
	 y(0) = -\frac{1}{\varphi''(0)}, \\
		 y'(0) = 0.
\end{array}
\right.
\end{equation}
The constant $\Theta$ is also given by 
\begin{equation}\label{theta2}
	\Theta = \frac{y'(L)}{\varphi''(0)}.
\end{equation}

%To establish our orbital stability result in the next section and since we do not have a smooth curve of periodic waves depending on the wave speed $\omega$ as in the case $a+b=0$, we need to do some useful simplifications in our approach. In fact, by assuming that $A_2 =0\in \mathbb{R}$ in \eqref{SystemEDO3}, we automatically have $A_1 = 0$. In this case, $a < 0$ is given as function of the modulus $k\in\left(0,\frac{1}{\sqrt{2}}\right)$ 
%\begin{equation}\label{relationa}
%a=- \frac{1}{16}\,{\frac {{L}^{2}}{ {\rm K} \left( k \right)^{2}\sqrt {{k}^{4}-{k}^{2}+1}}}<0,
%\end{equation}
%so that the smooth surface in Theorem $\ref{cnoidalcurve2}$ is a smooth curve depending on the modulus $k$. Next, by \eqref{relationa} we also have
%$
%b=\frac{1}{6}+ \frac{1}{16}\,{\frac {{L}^{2}}{ {\rm K} \left( k \right)^{2}\sqrt {{k}^{4}-{k}^{2}+1}}}>0.
%$

%\begin{remark}\label{relationsab}
%Notice that by considering $a+b = \tfrac{1}{6}$, we obtain the surface tension $\tau = 0$. In this case, parameters $a,b,c,d \in \mathbb{R}$ satisfy the equalities given in \eqref{relationabcd1} for $\theta^2 = \tfrac{2}{3} $. As consequence of the first identity in \eqref{relationabcd2}, we have $a=\tfrac{\nu}{6}$ and
%\begin{equation}\label{relation-nu}
%\nu= \nu_k=-\frac{3}{8}\frac{L^2}{{\rm K}(k)^2 \,\sqrt{k^4-k^2+1}}.
%\end{equation}
%\end{remark}

The focus now is to analyze the behaviour of the non-positive spectrum of $\mathcal{L}_3$. First, we see that it depends only on the parameter $k \in (0, 1)$  %$k \in \big(0, \sfrac{1}{\sqrt{2}}\big)$
and thus, to evaluate the sign of the constant $\Theta$ in order to guarantee that $\lambda = 0$ is the second eigenvalue, we need to use that the family of operators %$\left\{\mathcal{L}_3= \mathcal{L}_{3,k};\ k \in \left(0,\sfrac{1}{\sqrt{2}}\right) \right\}$
$\left\{\mathcal{L}_3= \mathcal{L}_{3,k};\ k \in (0,1) \right\}$ is isoinertial. To do so, we employ the use of  \textit{software Mathematica} to solve the initial value problem \eqref{PVI2}. The corresponding values of $\Theta$ are listed in the next tables.

\begin{table}[!h]
	\begin{tabular}{|c|c|}
		\hline 
		\multicolumn{2}{|c|}{$L=1$} \\ 
		\hline 
		$k$ & $\Theta$ \\
		\hline 
		$0.01$ & $-0.0507$ \\
		\hline
		$0.1$ & $-0.0211$ \\
		\hline 
		$0.3$ & $-0.0213$ \\
		\hline 
		$0.5$ & $-0.0230$ \\
		\hline 
		$0.7$ & $-0.0324$ \\
		\hline 
		$0.9$ & $-0.1227$ \\
		\hline 
		$0.99$ & $-3.6893$ \\
		\hline
			\end{tabular}
	\hspace{0.5cm} 
	\begin{tabular}{|c|c|}
		\hline 
		\multicolumn{2}{|c|}{$L=2\pi$} \\ 
		\hline 
		$k$ & $\Theta$ \\
		\hline 
		$0.01$ & $-10.0817$ \\
		\hline
		$0.1$ & $-5.2367$ \\
		\hline 
		$0.3$ & $-5.2873$ \\
		\hline 
		$0.5$ & $-5.7201$ \\
		\hline 
		$0.7$ & $-8.0607$ \\
		\hline 
		$0.9$ & $-30.4458$ \\
		\hline 
		$0.99$ & $-915.132$ \\
		\hline 
	\end{tabular}
	\hspace{0.5cm} 
	\begin{tabular}{|c|c|}
		\hline 
		\multicolumn{2}{|c|}{$L=50$} \\ 
		\hline 
	$k$ & $\Theta$ \\
	\hline 
	$0.01$ & $-2.96 \times 10^3$ \\
	\hline
	$0.1$ & $-2.63 \times 10^3$ \\
	\hline 
	$0.3$ & $-2.66 \times 10^3$ \\
	\hline 
	$0.5$ & $-2.88\times 10^3$ \\
	\hline 
	$0.7$ & $-4.06 \times 10^3$ \\
	\hline 
	$0.9$ & $-1.53 \times 10^4$ \\
	\hline
	$0.99$ & $-4.60 \times 10^5$ \\
	\hline 
	\end{tabular}
	\hspace{0.5cm} 
	\begin{tabular}{|c|c|}
		\hline 
		\multicolumn{2}{|c|}{$L=100$} \\ 
		\hline 
		$k$ & $\Theta$ \\
		\hline 
		$0.01$ & $-2.97 \times 10^3$ \\
		\hline
		$0.1$ & $-2.11 \times 10^3$ \\
		\hline 
		$0.3$ & $-2.13 \times 10^3$ \\
		\hline 
		$0.5$ & $-2.30 \times 10^3$ \\
		\hline 
		$0.7$ & $-3.24 \times 10^4 $ \\
		\hline  
		$0.9$ & $-1.22 \times 10^5$ \\
		\hline
		$0.99$ & $-3.68 \times 10^6$ \\
		\hline 
	\end{tabular}
\end{table}

\indent Thanks to Proposition \ref{isonertialP} and Remark $\ref{rem12}$, we see that ${\rm in}(\mathcal{L}_3) = (1,1)$, that is,
\begin{equation}
{\rm n}(\mathcal{L}_3) = 1 \quad \text{ and } \quad {\rm z}(\mathcal{L}_3) = 1.
\end{equation} Summarizing the results obtained above, we have the following lemma.

\begin{lemma}\label{lemma-L3}
	Let $(\varphi, \psi) = (\varphi_{k}, B \varphi_{k})$ be the pair of periodic solutions given by Proposition $\ref{cnoidalcurve2}$ with $A_1=A_2=0$. The operator $\mathcal{L}_3$ defined in \eqref{L1L2case2} has exactly one negative eigenvalue which is simple and zero is the second eigenvalue which is simple with associated eigenfunction $\varphi'$. Moreover, the remainder of the spectrum is constituted by a discrete set of eigenvalues.
	\end{lemma}
\begin{flushright}
	$\blacksquare$
\end{flushright}
\indent Gathering the results of Lemmas \ref{lemma-L4} and \ref{lemma-L3}, we are in position to establish the following theorem.

\begin{proposition}\label{simplekernel2}
Let $(\varphi, \psi) = (\varphi_{k}, B \varphi_{k})$ be the pair of periodic solutions given by Proposition $\ref{cnoidalcurve2}$ with $A_1=A_2=0$. The operator $\mathcal{L}$ defined in \eqref{operatorL} has exactly one negative simple eigenvalue and zero is a simple eigenvalue with associated eigenfunction $(\varphi', B \varphi')$. Moreover, the remainder of the spectrum is constituted by a discrete set of eigenvalues.
\end{proposition}
\begin{flushright}
	$\blacksquare$
\end{flushright}

\begin{remark}
We got some difficulties in studying the problem concerning the orbital stability concerning non-multiple solutions as in $(\ref{LIcn})$ since the spectral analysis in this case is too hard to handle. As far as we know, it is a hard task to  obtain a similar transformation as in \eqref{similartransformation1} in order to determine that ${\rm n}(\mathcal{L})=1$ and ${\rm z}(\mathcal{L})=1$. The multiple solutions $(\varphi,\psi)=(\varphi,B\varphi)$ are useful since $\mathcal{L}$ is similar to a diagonal matrix formed by well-known Hill operators with periodic potential. Moreover, we can study the inertial index of $\mathcal{L}$ using the non-multiple solution in $(\ref{LIcn})$ using the quadratic form associated to the operator $\mathcal{L}$ in \eqref{operatorL} given by
$$\begin{array}{llll}
\displaystyle \big(\mathcal{L}(f,g), (f,g)\big)_{L_{per}^2\times L_{per}^2}&= (\widetilde{\mathcal{L}}_1\, f,f)_{L_{per}^2}+(\widetilde{\mathcal{L}}_2\,g, g)_{L_{per}^2}\\
&\displaystyle+b \, \omega \int_0^L [(f'-g')(x)]^2dx+\int_0^L(\omega-\psi)(x)[(f-g)(x)]^2dx,
\end{array}$$
where $\widetilde{\mathcal{L}}_1=-(b\omega-c)(-\partial_x^2)+(1-\omega)-\psi$ and $\widetilde{\mathcal{L}}_2=-(b\omega-a)(-\partial_x^2)+(1-\omega)-\psi+\varphi$. Even in the case $b=d$, we are not able to determine the non-positive spectrum of $\mathcal{L}$ by studying the non-positive spectrum of $\widetilde{\mathcal{L}}_1$ and $\widetilde{\mathcal{L}}_2$. One of the reasons is that $(\varphi',\psi')\in \Ker(\mathcal{L})$ but we don't know how to determine $\Ker(\widetilde{\mathcal{L}}_i)$ for $i=1,2$.
\end{remark}

\section{Orbital Stability}\label{stability-section}

In this section, we present our orbital stability result. Before presenting our results, we need some useful notations. It is well known that (\ref{SystemEDO1}) is invariant under translations in the sense that if $V = (\eta, u) \in X$ is a solution for $(\ref{abcd})$, we have
\begin{equation}
{\rm T}_r V = (\eta(\cdot+r),u(\cdot+r)),
\end{equation}
is also a solution for all $r\in\mathbb{R}$.\\
\indent In what follows, let $L>0$ be fixed. We now present our notion of orbital stability as in \cite{GrillakisShatahStraussI}.

\begin{definition}\label{defstab}
 The periodic wave $\Psi=(\varphi,B\varphi)$ is orbitally stable in $X$ if for all $\varepsilon > 0$, there exists $\delta > 0$ with the following property: if
 \begin{equation*}
 	\|(\eta_0, u_0) - (\varphi, \psi)\|_X < \delta,
 \end{equation*}
and $V(t) = (\eta(t), u(t))$ is a solution of \eqref{abcd} in some interval $[0,t_0)$ with $V(0)=(\eta_0,u_0)$, then $V(t)$ can be continued to a solution in $[0,+\infty)$ and
\begin{equation*}
	 \sup_{t\geq0}\inf_{r \in \mathbb{R}} \|V(t) - {\rm T}_r \Psi\|_X < \varepsilon.
\end{equation*}
Otherwise, $\Psi$ is said to be orbitally unstable. 
\end{definition}

%since the well-posedness of the problem \eqref{abcd} is out of the scope of this paper, we assume the periodic Cauchy problem associated with \eqref{abcd} is globally well-posed in $X$. 

%In this section, with the intention of showing that the solution $\Psi = (\varphi, \psi) = (\varphi, B \varphi)$ obtained in Theorems \ref{cnoidalcurve1} and \ref{cnoidalcurve2} is orbitally stable when $a=c < 0$, $b=d>0$ considering the subcases $a+b = 0$ and $a+b = \tfrac{1}{6}$, we will use the main result of \cite{CristofaniPastor} which gives us sufficient conditions to guarantee the orbital stability of the cnoidal wave solution $\Psi = (\varphi, \psi)$ of the system \eqref{SystemEDO1}. 

\subsection{Case 1: $a+b = 0$.}\label{firstcase} Proof of Theorem $\ref{mainT}$-{\rm (i)}.\\
\indent We follow the arguments in \cite[Theorem 1.3]{andrade}. In fact, the construction of periodic waves has been determined in Proposition $\ref{cnoidalcurve1}$ and the spectral analysis was established in Proposition $\ref{simplekernel1}$. Both two facts give us assumption (H1). To prove (H2), let us consider the conserved functional
\begin{equation}\label{conserQ}
	Q(\eta, u )= F(\eta, u ) + (\partial_\omega A_2) M_1(\eta, u)+(\partial_\omega A_2) M_2(\eta, u),
\end{equation}
where $F$, $M_1$ and $M_2$ are conserved quantities defined in \eqref{F} and \eqref{M1M2}.

\begin{remark}\label{rem1234} It is important to be highlighted that in \cite{andrade} the authors construct periodic waves depending on the modulus $k\in(0,1)$. In our case, the periodic waves are depending on three parameters $(b,\omega,k)$, so that in assumption (H4) of \cite{andrade} the terms where appears the derivative in terms of $k$ need to be replaced by the derivative with respect to $\omega$. This fact gives us that $Q$ in $(\ref{conserQ})$ is suitable to be used to our stability approach. In addition, a rudimentary calculation shows that $2F(\varphi,\psi)+(\partial_\omega A_2) M_1(\varphi, \psi)+(\partial_\omega A_2) M_2(\varphi, \psi)\neq0$ and condition (H4) in \cite{andrade} is satisfied.
	\end{remark}
	
Let us turn back to the proof of the theorem. Indeed, we see that
\begin{eqnarray}\label{Phi1}
	Q'(\Psi) = Q'(\varphi, \psi)= \begin{pmatrix}
		\psi-b\psi'' + \partial_\omega A_2 \\ 
		\varphi-b\varphi''+\partial_\omega A_1
	\end{pmatrix},
\end{eqnarray}
where $A_1, A_2$ are given by \eqref{A1case1} and \eqref{A2case1}, respectively. Thus, if we define
$
\Phi= \frac{\partial}{\partial \omega}(\varphi, \psi)= \left( \begin{array}{l}  \partial_\omega \varphi \\ \partial_\omega  \psi
\end{array}
\right),
$
with $\psi= B \varphi$ and $B$ given in \eqref{Bcase1}, we see that $\mathcal{L}\Phi=Q'(\Psi)$. In addition, we obtain since $\mathcal{L}$ is a self-adjoint operator and ${\rm Ker}(\mathcal{L}) = [(\varphi',\psi')]$ that
$
	(\mathcal{L} \Phi, \Psi')_{L_{per}^2\times L_{per}^2} = (\mathcal{L}\, \Phi, (\varphi', \psi'))_{L_{per}^2\times L_{per}^2}=(\Phi, \mathcal{L}(\varphi',\psi'))_{L_{per}^2\times L_{per}^2}=0.
$ To conclude the orbital stability, it remains to prove  $I=\big(\mathcal{L}\Phi, \Phi \big)_{L_{per}^2\times L_{per}^2}<0$. Indeed, we have
\begin{equation}\label{quantityIcase1}
I  =  \frac{\partial}{\partial \omega} \int_0^L \varphi \psi \; dx- b \frac{\partial}{\partial \omega} \int_0^L \psi'' \varphi \; dx+ I_1+I_2= I_0+I_1+I_2,
\end{equation}
where
$$
I_0= \frac{\partial}{\partial \omega} \left( B\, \|\varphi\|_{L_{per}^2}^2+b\, B\,\|\varphi'\|_{L_{per}^2}^2\right),  \quad   I_1=( \partial_\omega A_2, \partial_{\omega} \varphi)_{L_{per}^2} \quad \text{and} \quad  I_2=( \partial_\omega A_1, \partial_{\omega} \psi)_{L_{per}^2}.
$$

Using the explicit form of the cnoidal solution in \eqref{cnoidalsolution} for the pair $(\varphi, \psi)=(\varphi, B \varphi)$, we obtain
\begin{eqnarray}
I_0=\frac{\partial}{\partial \omega} \left( BLb_0^2+2Bb_0b_2 J_1    +Bb_2^2 J_2+\frac{16\,b_2^2 \, b B}{L^2}{\rm K}(k)^2 J_3         \right),
\end{eqnarray}
where
\begin{equation}\label{J1J2J3}
\begin{array}{llll}
 & &\displaystyle J_1= \int_0^L {\rm cn}^2\left(\frac{2{\rm K}(k)x}{L}, k\right)\; dx, \quad J_2= \int_0^L {\rm cn}^4\left(\frac{2{\rm K}(k)x}{L}, k\right)\; dx, \\\\
& & \displaystyle J_3=\int_0^L {\rm cn}^2\left(\frac{2{\rm K}(k)x}{L}, k\right){\rm dn}^2\left(\frac{2{\rm K}(k)x}{L}, k\right){\rm sn}^2\left(\frac{2{\rm K}(k)x}{L}, k\right)\; dx.
\end{array}
\end{equation}
Also,
\begin{equation}\label{I1+I2}
I_1+I_2= \left(Lb  \partial_\omega b_0+  B J_1 \partial_\omega b_2 +L b_0 \partial_\omega B + b_2 J_1  \partial_\omega B \right) \partial_\omega A_1 + \left(L \partial_\omega b_0+ J_1 \partial_\omega b_2 \right)\partial_\omega A_2.
\end{equation}

By using \cite[Formulas (312.02), (312.04) and (361.04)]{ByrdFriedman}, it is possible to calculate $J_i$, $i = 1,2,3$ in \eqref{J1J2J3}, as well as, the explicit parameters $b_0$, $b_2$, $B$, $A_1$ and $A_2$ given in \eqref{b0case1}-\eqref{A2case1} to determine the sign of $I$. Using numerics, we can also deduce the existence of
$
	b^* \simeq \frac{L^2}{591.94} > 0
$ such that for all $b > b^*$, we have
$
	I = (\mathcal{L}\Phi, \Phi)_{L_{per}^2\times L_{per}^2} < 0.
$

In the next figures, we have two set of plots which give us the behaviour of $I$ for $B>0$. In all plots of type (a), we obtain the plot of the quantity $I$ by fixing $L$, $b$ and $\omega$, so that $I$ is given only in terms of $k$. In all graphics of type (b), we give the behaviour of $I$ in terms of $\omega$ by fixing $L$, $b$ and $k$. In all graphics of type (c), we plot the $3$-D graphic in terms of $\omega$ and $k$ by fixing $L$ and $b$. As far as we can see, the behaviour remains the same all three graphics of type (a), (b) and (c). The important fact is that in all graphics, we can conclude $I<0$ for all $b > b^*$, $\omega\in (-1,1)$ and $k\in\big(0,\sfrac{1}{\sqrt{2}}\big).$

%\begin{figure}[!h]
%	\subfigure[$b=0.001 : \omega=-0.5$]{
%		\includegraphics[width=4cm,height=5cm]{L=1-b=0.001-w=-0.5-k.eps}}
%	\quad
%	\subfigure[$b=0.001 : k = 0.1$]{
%		\includegraphics[width=4cm,height=5cm]{L=1-b=0.001-k=0.1-w.eps}}
%	\quad
%	\subfigure[$b = 0.001$]{
%		\includegraphics[width=4cm,height=5cm]{L=1-b=0.001-3D.png}}
%	\caption{Graphics of the quantity $I$ for $B>0$, $L=1$ and $b < b^*$.}
%	\label{figura1}
%\end{figure}

\begin{figure}[!h]
	\subfigure[$b=0.002 : \omega=0$]{
		\includegraphics[width=4cm,height=5cm]{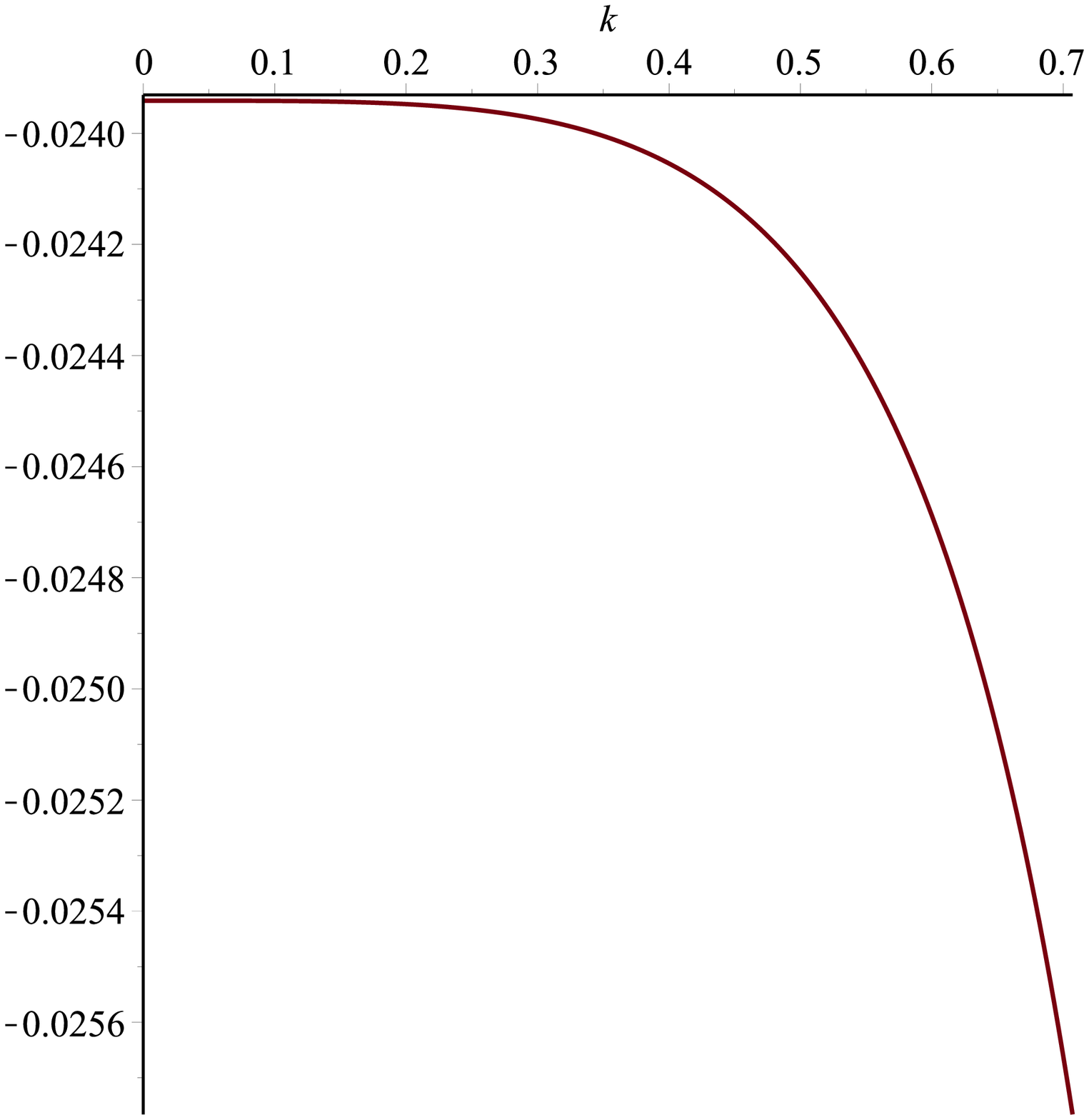}}
	\quad
	\subfigure[$b=0.002 : k = 0.3$]{
		\includegraphics[width=4cm,height=5cm]{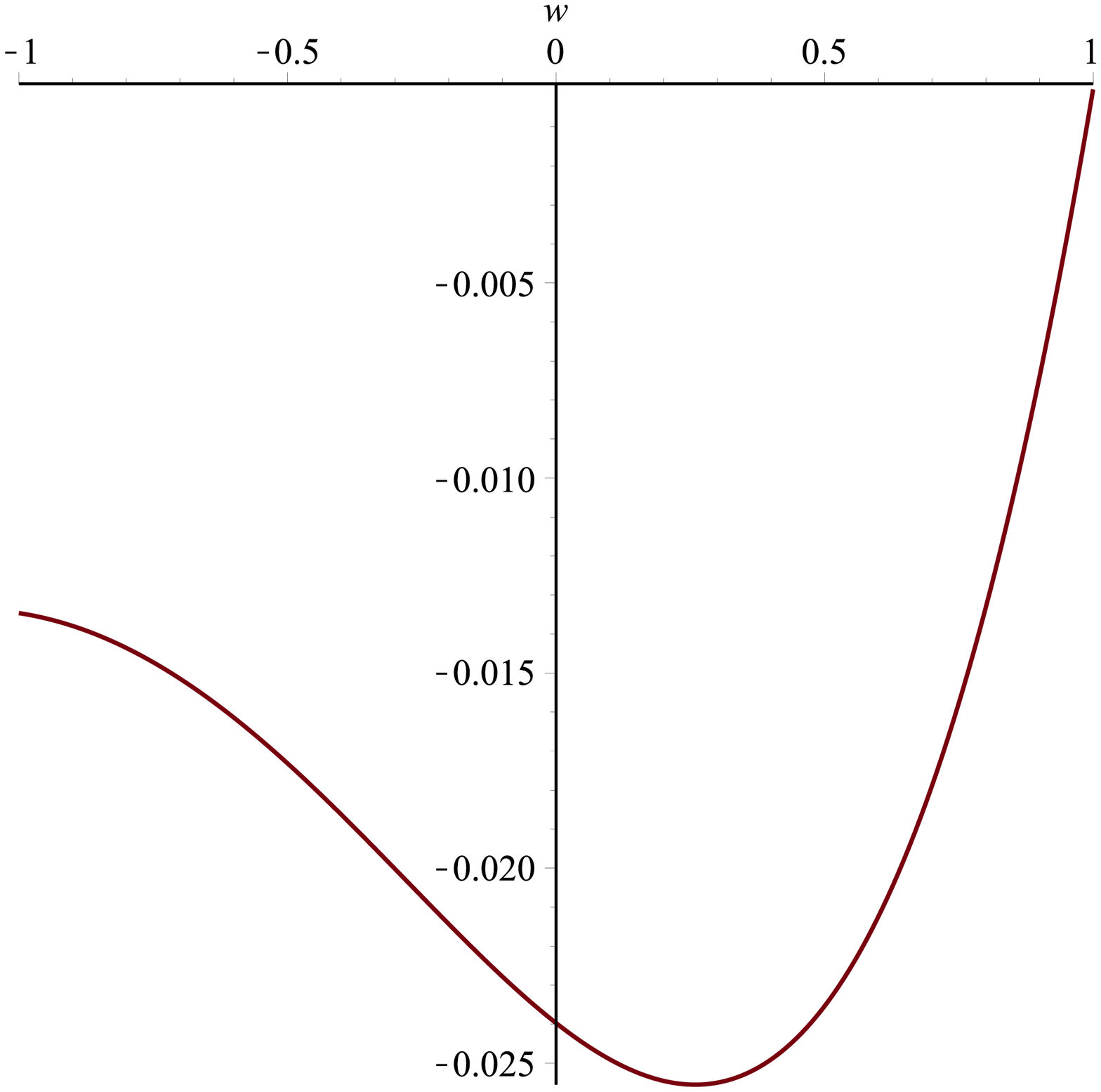}}
	\quad
	\subfigure[$b = 0.002$]{
		\includegraphics[width=4cm,height=5cm]{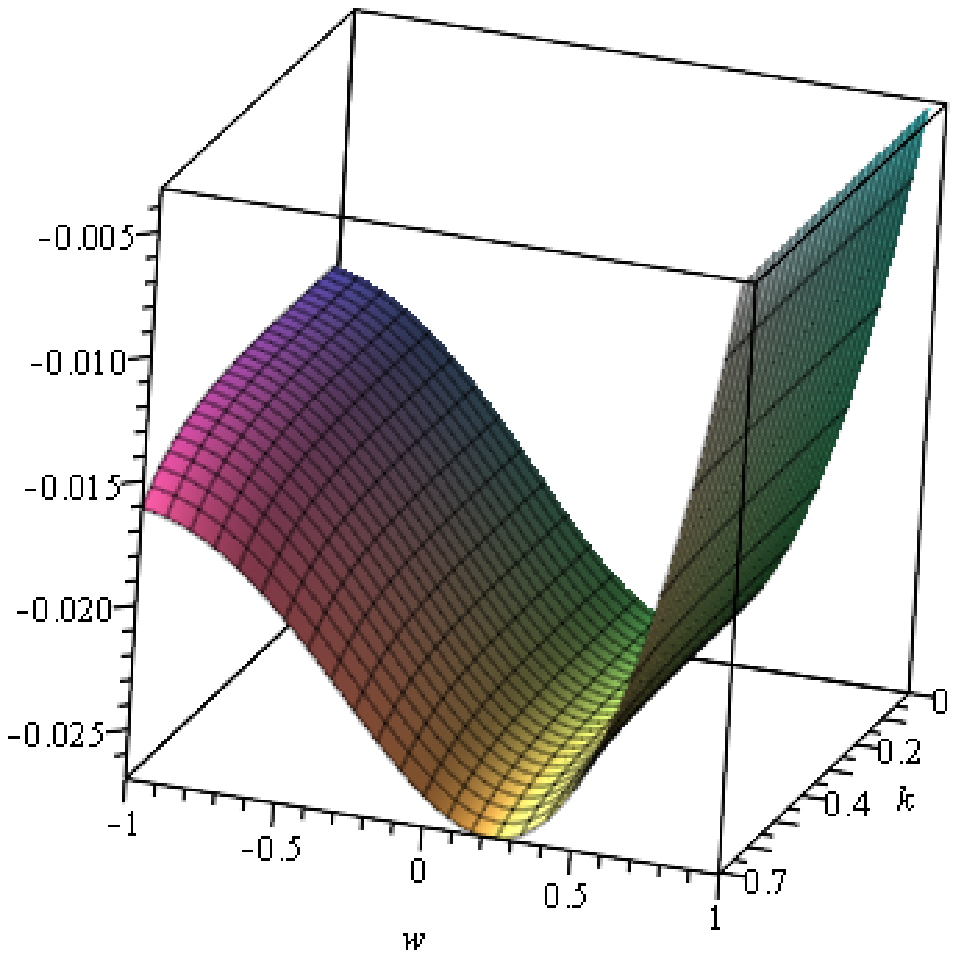}}
	\caption{Graphics of the quantity $I$ for $B>0$ and $L=1$.}
	\label{figura2}
\end{figure}

%\begin{figure}[!h]
%	\subfigure[$b=0.05 : \omega=0.25$]{
%		\includegraphics[width=4cm,height=5cm]{L=2Pi-b=0.05-w=0.25-k.eps}}
%	\quad
%	\subfigure[$b=0.05 : k = 0.7$]{
%		\includegraphics[width=4cm,height=5cm]{L=2Pi-b=0.05-k=0.7-w.eps}}
%	\quad
%	\subfigure[$b = 0.05$]{
%		\includegraphics[width=4cm,height=5cm]{L=2Pi-b=0.05-3D.png}}
%	\caption{Graphics of the quantity $I$ for $B>0$, $L=2\pi$ and $b < b^*$.}
%	\label{figura50}
%\end{figure}

\begin{figure}[!h]
	\subfigure[$b=0.07 : \omega=0.5$]{
		\includegraphics[width=4cm,height=5cm]{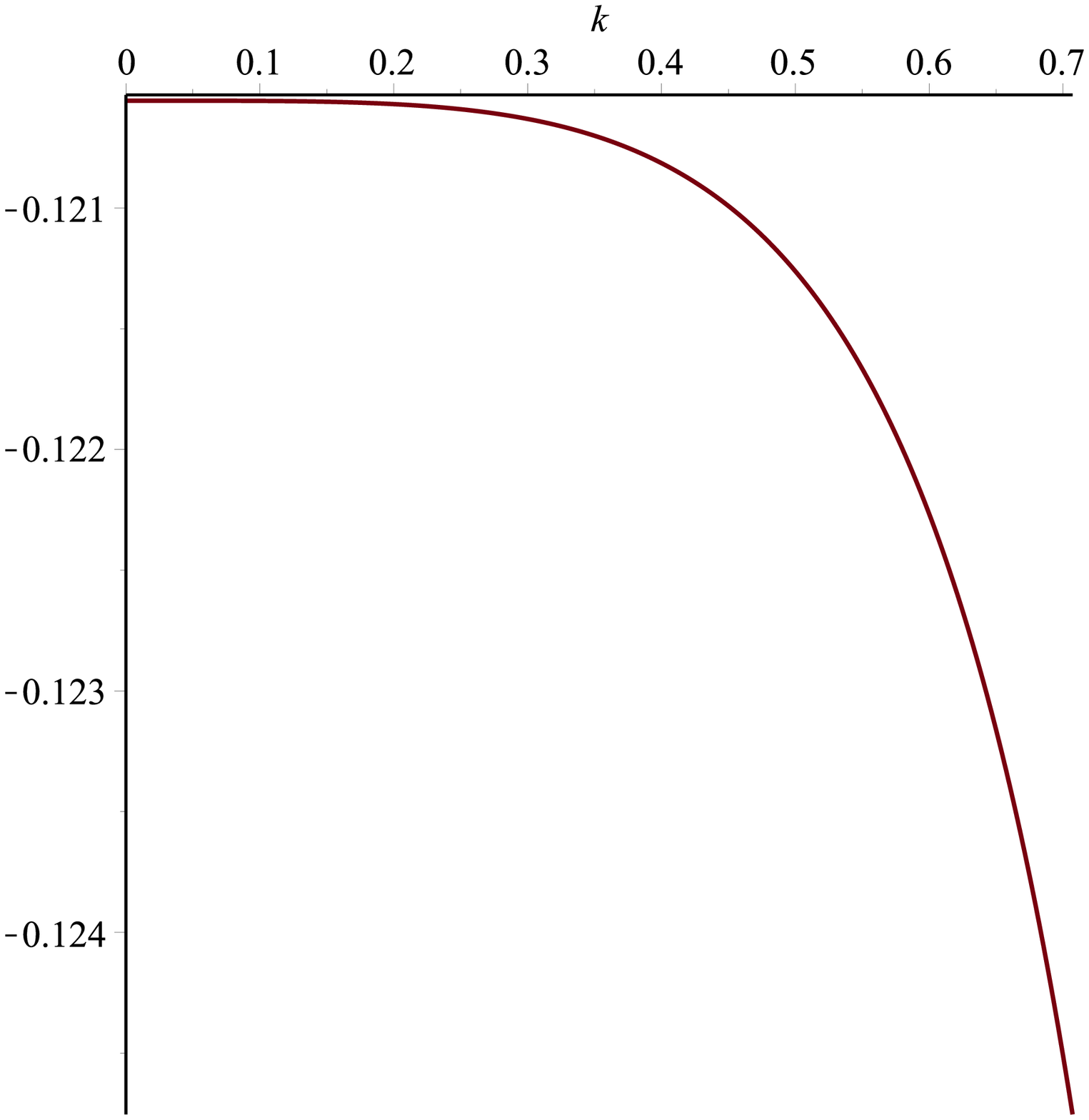}}
	\quad
	\subfigure[$b=0.07 : k = 0.5$]{
		\includegraphics[width=4cm,height=5cm]{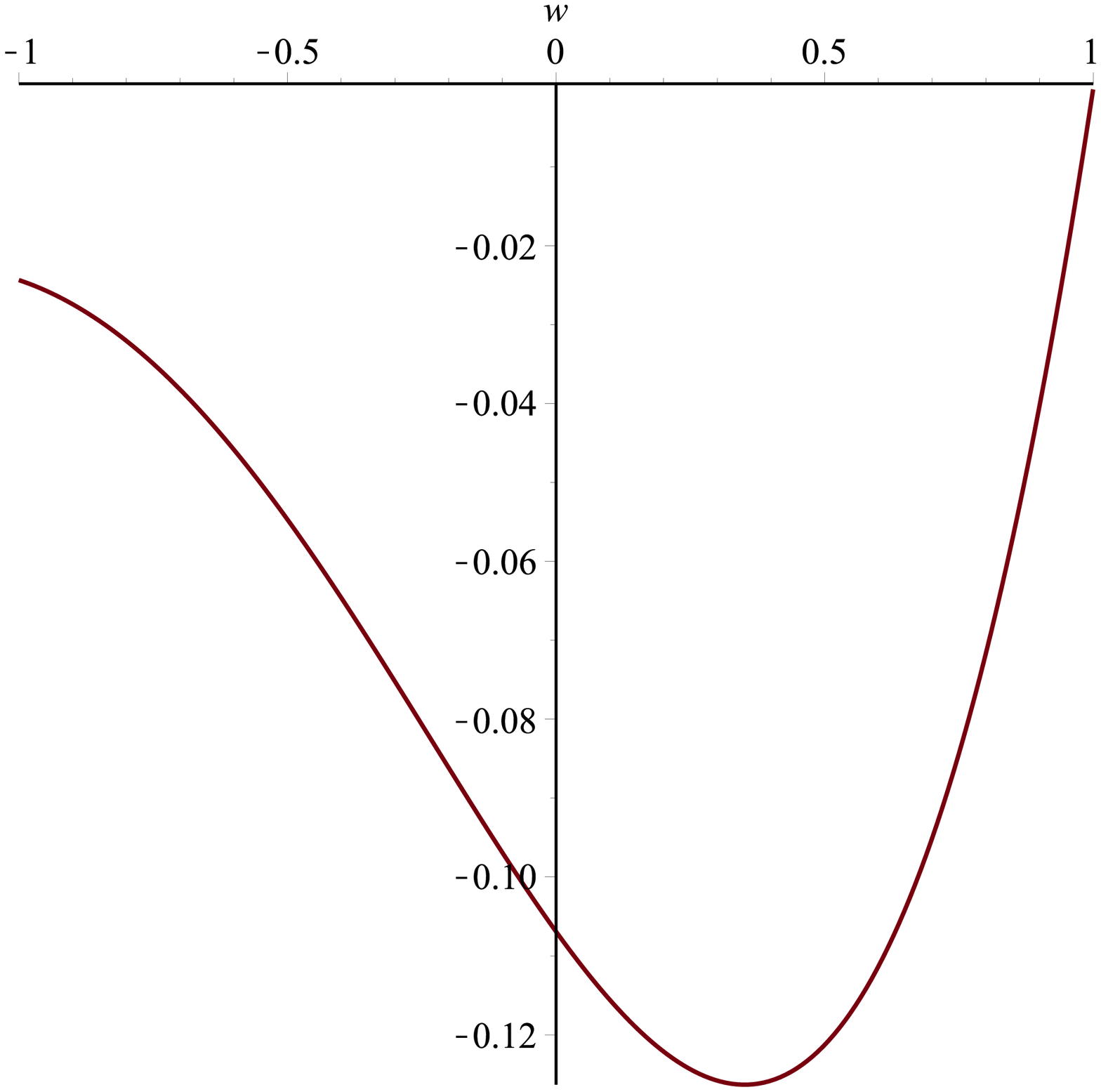}}
	\quad
	\subfigure[$b = 0.07$]{
		\includegraphics[width=4cm,height=5cm]{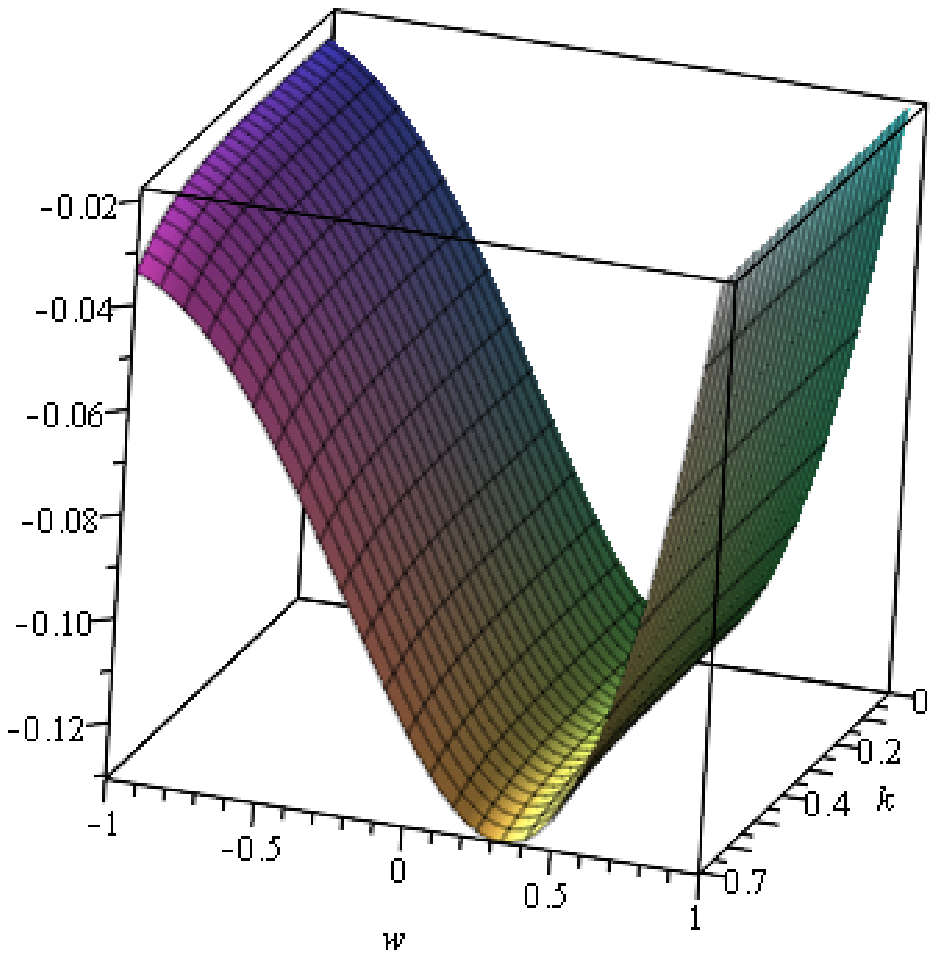}}
	\caption{Graphics of the quantity $I$ for $B>0$ and $L=2\pi$.}
	\label{figura5}
\end{figure}

\begin{figure}[!h]
	\subfigure[$b=4.23 : \omega=-0.5$]{
		\includegraphics[width=4cm,height=5cm]{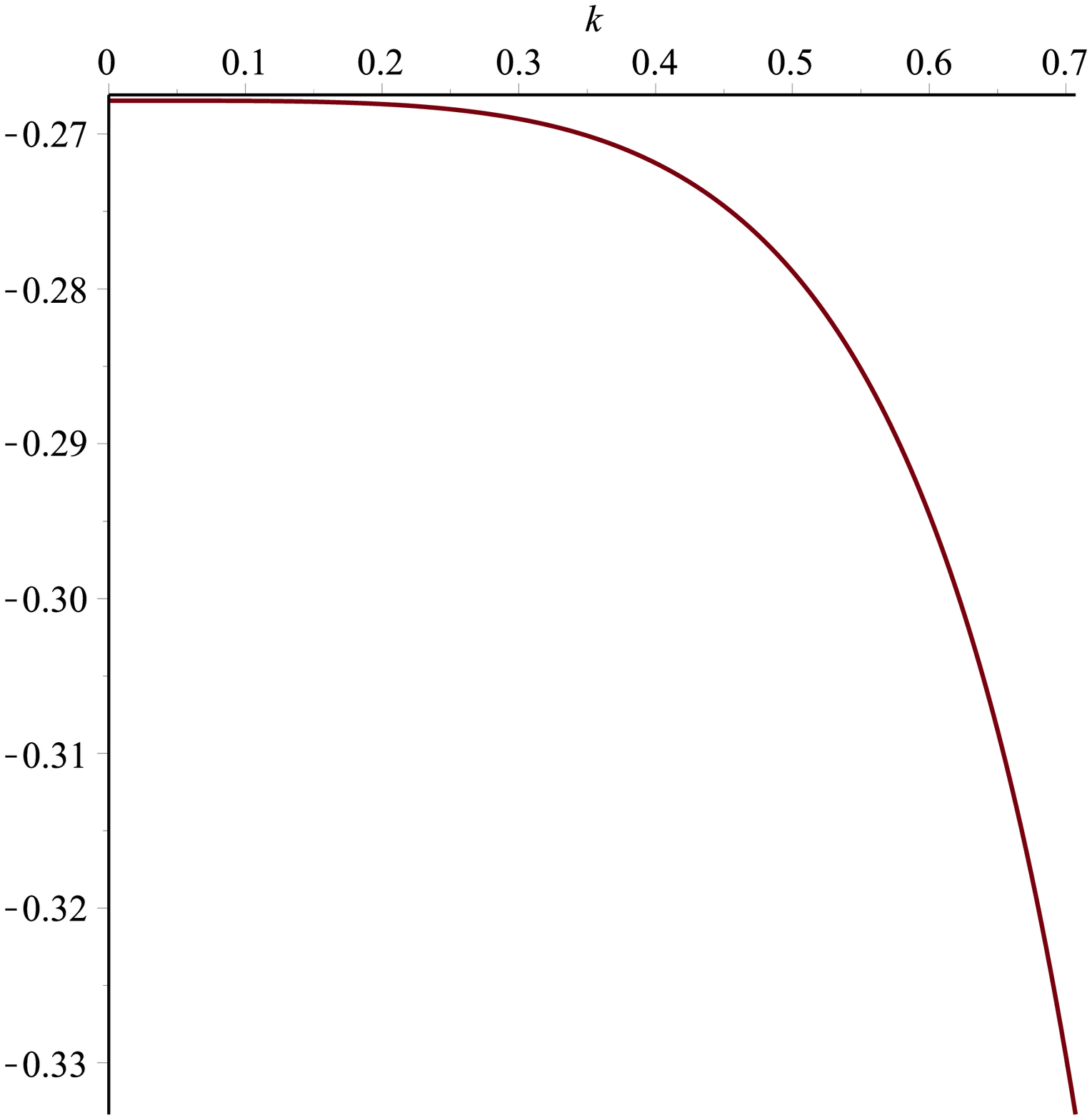}}
	\quad
	\subfigure[$b=4.23 : k = 0.01$]{
		\includegraphics[width=4cm,height=5cm]{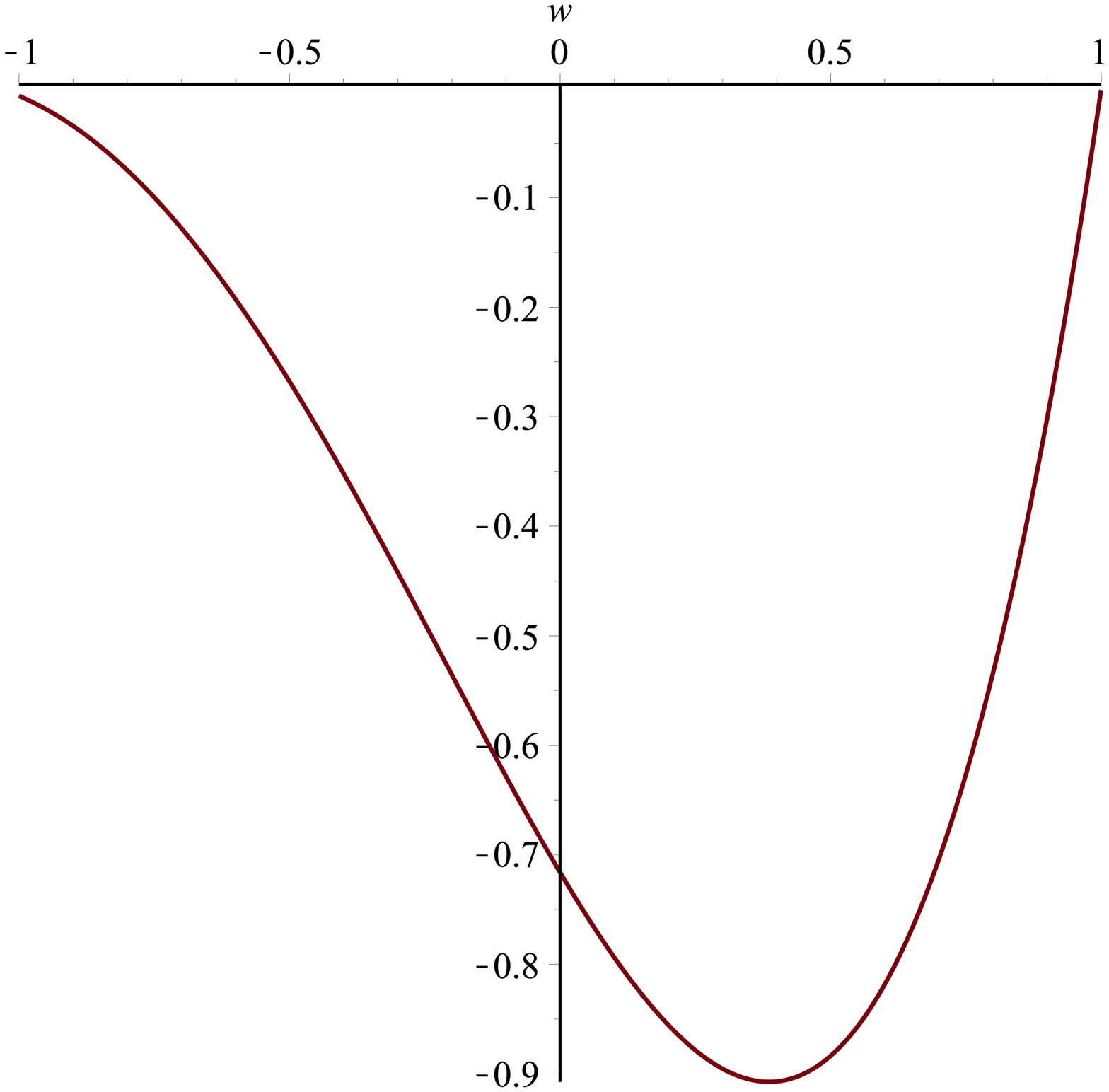}}
	\quad
	\subfigure[$b = 4.23$]{
		\includegraphics[width=4cm,height=5cm]{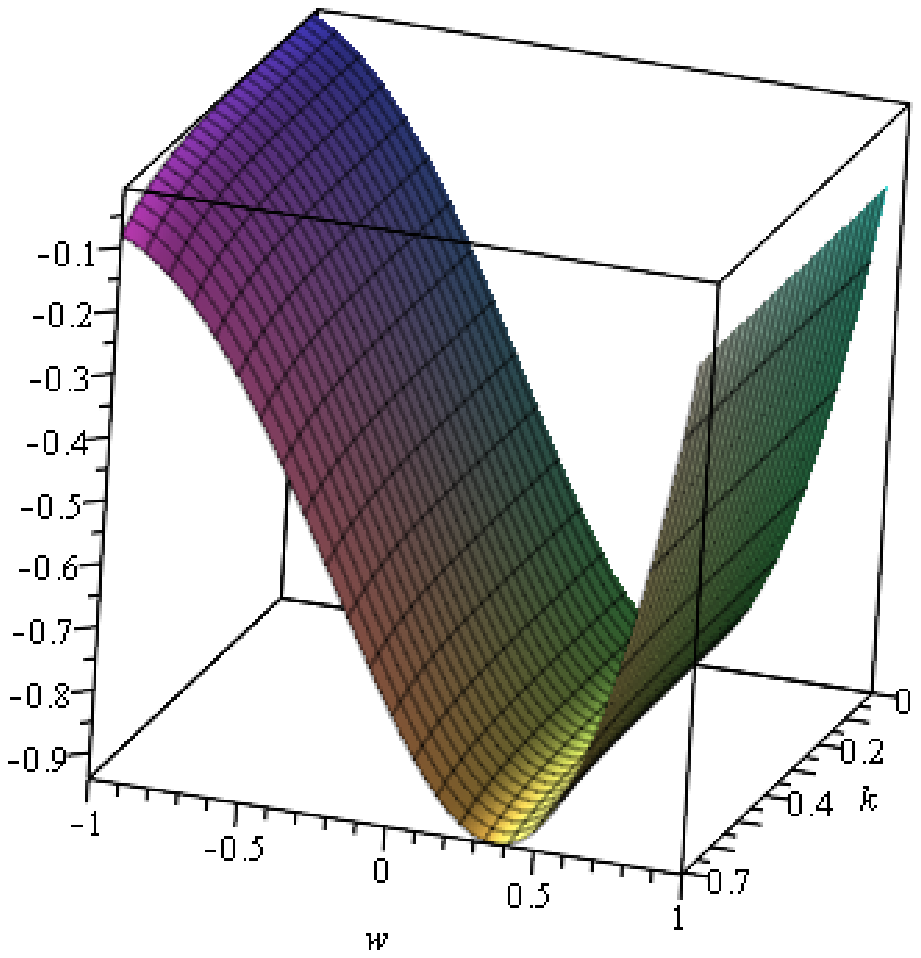}}
	\caption{Graphics of the quantity $I$ for $B>0$ and $L=50$.}
	\label{figura50}
\end{figure}

In the case of $B < 0$, we conclude by analysing the three graphics in Figures \ref{figura5} and \ref{figura13} the existence of a certain symmetry with the case $B>0$, so that there is no interference on the sign of the quantity $I$ in \eqref{quantityIcase1}.

\begin{figure}[!h]
	\subfigure[$b=0.07 : \omega=0.5$]{
		\includegraphics[width=4cm,height=5cm]{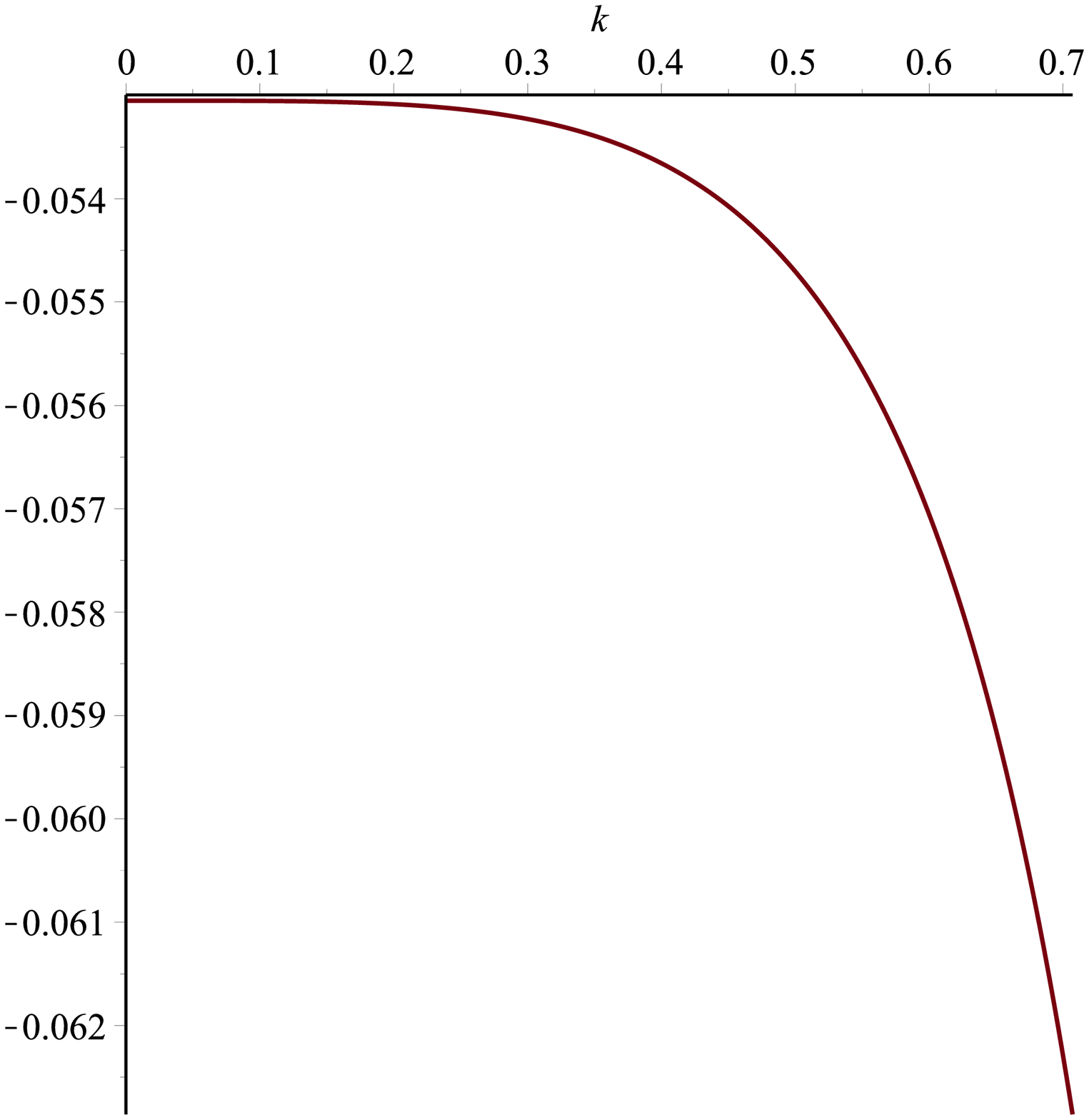}}
	\quad
	\subfigure[$b=0.07 : k = 0.5$]{
		\includegraphics[width=4cm,height=5cm]{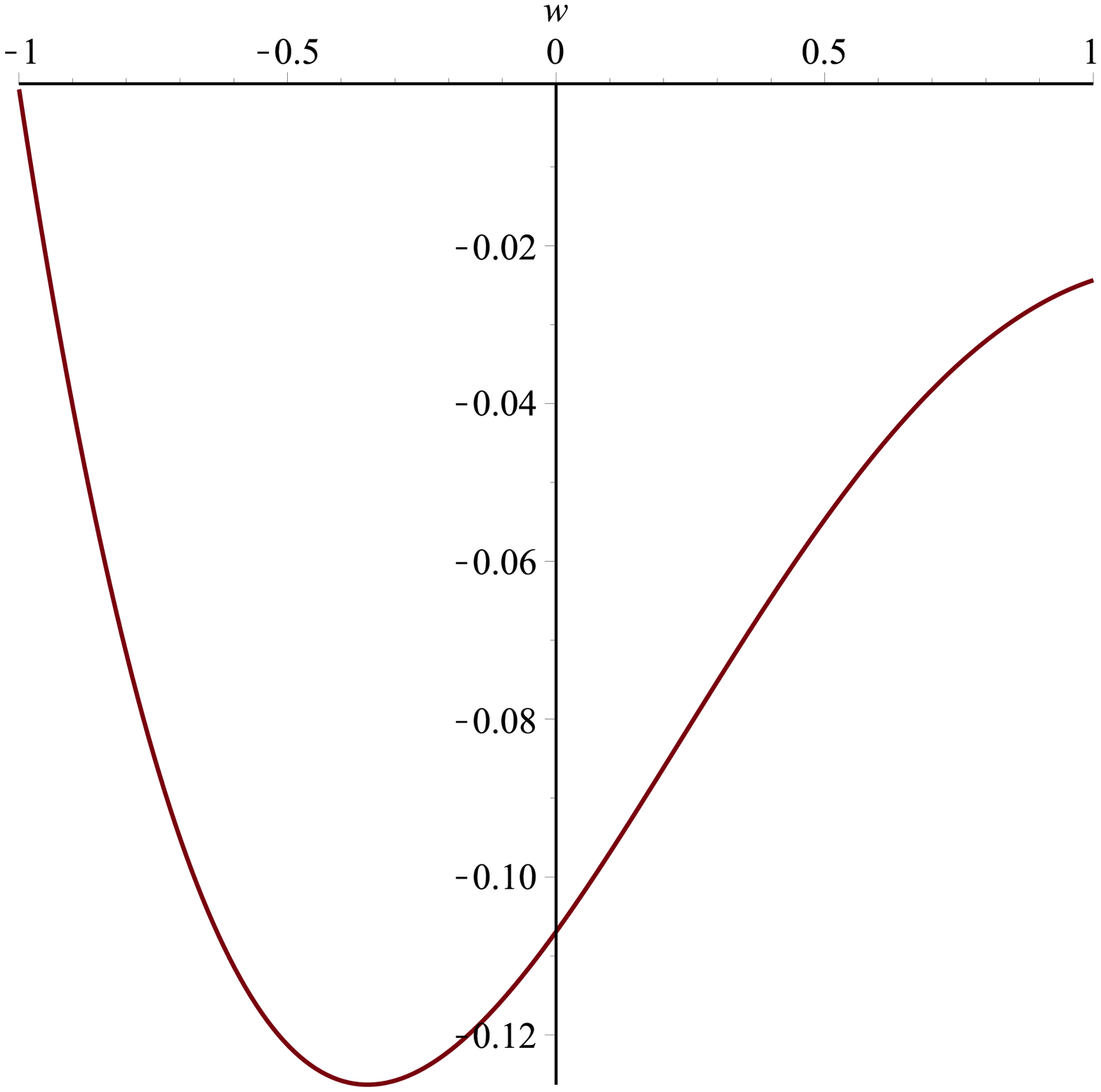}}
	\quad
	\subfigure[$b = 0.07$]{
		\includegraphics[width=4cm,height=5cm]{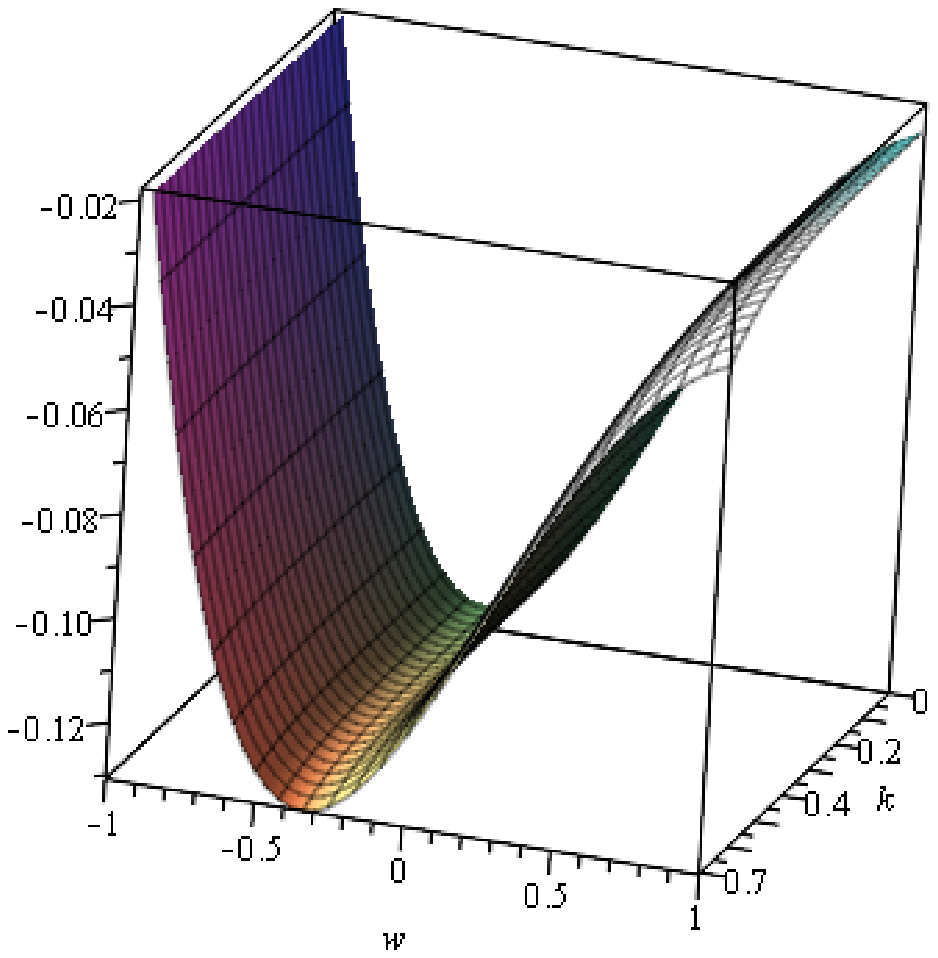}}
	\caption{Graphics of the quantity $I$ for $B<0$ and $L=2\pi$.}
	\label{figura13}
\end{figure}
\begin{flushright}
	$\blacksquare$
\end{flushright}

\subsection{Case 2: $a+b = \tfrac{1}{6}$.} Proof of Theorem $\ref{mainT}$-(ii).\\
\indent Since $A_1=A_2=0$, the orbital stability in this case can be determined directly by the stability theorem in \cite{GrillakisShatahStraussI} (see also \cite[Theorem 1.3]{andrade}). The construction of the periodic wave has been determined in Proposition $\ref{cnoidalcurve2}$ and the spectral properties concerning the linearized operator $\mathcal{L}$ in $(\ref{matrixop})$ for this case was established in Proposition $\ref{simplekernel2}$. Both two facts give assumption (H1) and it remains to establish assumption (H2). In fact, since the pair $(\varphi,B\varphi)$ of cnoidal waves of the form $(\ref{cnoidalsolution})$ are solutions of $(\ref{SystemEDO2})$ only if $\omega=0$, we can not construct a smooth curve of periodic waves depending on the wave speed $\omega$ as in the first case. This fact brings some difficulties in order to obtain $\Phi$ as requested in assumption (H2). To overcome this problem, since $A_1=A_2=0$ we need to consider $Q(\eta, u )= F(\eta, u )$ where $F$ is a conserved quantity defined in \eqref{F}. Also,  since by Proposition \ref{simplekernel2} we have that ${\rm Ker}(\mathcal{L})=[(\varphi', \psi')]$, one has from the fact $\mathcal{L}$ is a self-adjoint operator that ${\rm Ker}(\mathcal{L})^{\perp}={\rm R}(\mathcal{L})$. This fact allows to guarantee the existence of a unique $\Phi$ satisfying $\mathcal{L}\Phi=\begin{pmatrix}
			\psi-b\psi''  \\ 
			\varphi-b\varphi''
		\end{pmatrix}= Q'(\varphi, \psi)$ %In addition, we infer
%$
%		Q'(\varphi, \psi)= F'(\varphi, \psi) = \begin{pmatrix}
%			\psi-b\psi''  \\ 
%			\varphi-b\varphi''
%		\end{pmatrix}
%$
and since $\mathcal{L}$ is a self-adjoint operator with ${\rm Ker}(\mathcal{L}) = [(\varphi',\psi')]$, we obtain that
$
	(\mathcal{L} \Phi, \Psi')_{L_{per}^2\times L_{per}^2} = (\mathcal{L}\, \Phi, (\varphi', \psi'))_{L_{per}^2\times L_{per}^2}=(\Phi, \mathcal{L}(\varphi', \psi'))_{L_{per}^2\times L_{per}^2}=0.
$

\indent On the other hand, using the relation \eqref{STorthogonal}, we can write 
\begin{eqnarray*}
	\big( \mathcal{L}\Phi, \Phi \big)_{L_{per}^2\times L_{per}^2} & = & \big( \big((1-b\partial_x^2)\psi, (1-b\partial_x^2)\varphi\big), \Phi \big)_{L_{per}^2\times L_{per}^2} \nonumber \\
	& = &  \big( \mathcal{L}^{-1}\big((1-b\partial_x^2)\psi, (1-b\partial_x^2)\varphi\big), \big((1-b\partial_x^2)\psi, (1-b\partial_x^2)\varphi \big) \big)_{L_{per}^2\times L_{per}^2} \nonumber \\
	& = & \left( \begin{pmatrix}
		\mathcal{L}_3 & 0 \\ 
		0 & \mathcal{L}_4
	\end{pmatrix}^{-1} (\mathcal{S}T) \begin{pmatrix}
		B \\ 
		1
	\end{pmatrix}(1-b \partial_x^2)\varphi, (\mathcal{S}T)\begin{pmatrix}
		B \\ 
		1
	\end{pmatrix}(1-b\partial_x^2)\varphi \right)_{L_{per}^2\times L_{per}^2}.
\end{eqnarray*}

Since $B =\pm \sqrt{2}$, we have that
\begin{equation}\label{STproduct}
	(\mathcal{S}T)\begin{pmatrix}
		B \\ 
		1
	\end{pmatrix} = \begin{pmatrix}
		2\sqrt{\frac{2}{3}} \\ 
		-\frac{1}{\sqrt{3}}
	\end{pmatrix}.
\end{equation}
where, depending on the value for $B$, the matrix $\mathcal{S}$ is given by \eqref{matrixScase2} ($B = \sqrt{2}$) or by \eqref{matrixScase2-2} ($B=-\sqrt{2}$). In addition, the negative parameters $b_0, b_2$ as in $(\ref{b0-case2-w=0})$ and $(\ref{b2-case2-w=0})$, respectively, does not depend on the sign of $B$. Thus, considering $f_0 = (1-b\partial_x^2)\varphi$, we obtain that the value of $I$ is given by
\begin{equation}\label{QI}
I=\big( \mathcal{L}\Phi, \Phi \big)_{L_{per}^2\times{L_{per}^2}}=\frac{8}{3} (\mathcal{L}_3^{-1}\, f_0, f_0)_{L_{per}^2}+ \frac{1}{3}(\mathcal{L}_4^{-1}\, f_0, f_0)_{L_{per}^2}=\frac{8}{3}\mathcal{J}_1+\frac{1}{3}\mathcal{J}_2.
\end{equation}

%\subsubsection{Case $B=-\sqrt{2}$.}\label{-sqrt(2)}
%
%As in \cite[Subsections 4.2.1 and 4.2.2]{HakkaevStanislavovaStefanov}, considering $B=-\sqrt{2}$ and by using the relation \eqref{STorthogonal}, we can write
%\begin{eqnarray*}
%\big( \mathcal{L}(\Phi), \Phi \big)_{\mathbb{L}^2} & = & \big( \big((1-b\partial_x^2)\psi, (1-b\partial_x^2)\varphi\big), \Phi \big)_{\mathbb{L}^2} \nonumber \\
%& = &  \big( \mathcal{L}^{-1}\big((1-b\partial_x^2)\psi, (1-b\partial_x^2)\varphi\big), \big((1-b\partial_x^2)\psi, (1-b\partial_x^2)\varphi \big) \big)_{\mathbb{L}^2} \nonumber \\
%& = & \left( \begin{pmatrix}
%\mathcal{L}_3 & 0 \\ 
%0 & \mathcal{L}_4
%\end{pmatrix}^{-1} (ST) \begin{pmatrix}
%-\sqrt{2} \\ 
%1
%\end{pmatrix}(1-b \partial_x^2)\varphi, (ST)\begin{pmatrix}
%-\sqrt{2} \\ 
%1
%\end{pmatrix}(1-b\partial_x^2)\varphi \right)_{\mathbb{L}^2}
%\end{eqnarray*}
%and, hence, 
%\begin{equation}\label{QI}
%\mathcal{I}:=\big( \mathcal{L}(\Phi), \Phi \big)_{\mathbb{L}^2}=\frac{1}{3} \left( 8 (\mathcal{L}_3^{-1}\, f_0, f_0)_{L^2_{per}}+(\mathcal{L}_4^{-1}\, f_0, f_0)_{L^2_{per}} \right)
%\end{equation}
%where $f_0:= (1-b\partial_x^2)\varphi$ and we have used
%\begin{equation}\label{STproduct}
%ST\begin{pmatrix}
%-\sqrt{2} \\ 
%1
%\end{pmatrix} = \begin{pmatrix}
%2\frac{\sqrt{2}}{\sqrt{3}} \\ 
%-\frac{1}{\sqrt{3}}
%\end{pmatrix}.
%\end{equation}

To calculate the value of $I$ in $(\ref{QI})$, we need to use similar arguments as in \cite[Subsection 4.2.1]{HakkaevStanislavovaStefanov}. Indeed, it is clear that 
$\mathcal{L}_3 \varphi'= a \varphi'''+\varphi'+2\varphi \varphi'.$
Since $A_1=A_2=0$, one has
\begin{equation}\label{QI2}
a\varphi''+\varphi+\varphi^2=0.
\end{equation}
\indent Taking the derivative with respect to $a$ in \eqref{QI2}, we see that $\mathcal{L}_3^{-1}\, \varphi''=- \partial_a\varphi$. From this fact and \eqref{QI2}, we obtain $\mathcal{L}_3\, \varphi= \varphi^2=-a\varphi''-\varphi$ and thus $\mathcal{L}_3^{-1}\, \varphi=a \partial_a \varphi-\varphi$. We have
\begin{equation}\label{QI5}
(\mathcal{L}_3^{-1}\, f_0, f_0)_{L_{per}^2}=(a+b)(\partial_a\varphi, \varphi)_{L_{per}^2}-b(a+b)(\partial_a\varphi, \varphi'')_{L_{per}^2}-(\varphi, \varphi)_{L_{per}^2}+b(\varphi, \varphi'')_{L_{per}^2}.
\end{equation}

We need to handle with all inner products present in $(\ref{QI5})$. In fact, by \eqref{relationa} and the chain rule, it follows that
\begin{eqnarray*}
	 (\partial_a\varphi, \varphi)_{L^2_{per}}&=& \frac{1}{2} \frac{d}{da}\int_0^L \varphi^2(x)\; dx \\
& = & \left(\frac{dk}{da}\right)^{-1} \left[ \frac{L}{2}\frac{d}{dk}(b_0^2)+b_0b_2\frac{d}{dk}(J_1)+J_1 \frac{d}{dk}(b_0b_2) +\frac{b_2^2}{2}\frac{d}{dk}(J_2)+\frac{1}{2} J_2 \frac{d}{dk}(b_2^2)\right],
\end{eqnarray*}
where $J_1, J_2$ are given by \eqref{J1J2J3}. Moreover,
\begin{eqnarray*}
(\partial_a\varphi, \varphi'')_{L_{per}^2}&=&-\frac{1}{2} \frac{d}{da}\int_0^L \varphi'(x)^2\; dx = -\frac{1}{2}\frac{d}{da}\left[ \frac{16\,b_2^2\,{\rm K}(k)^2}{L^2}J_3 \right] \nonumber \\
&=& -\frac{1}{2} \frac{16b_2^2{\rm K}(k)^2}{L^2}\frac{d}{da}(J_3) -\frac{J_3}{2}\frac{d}{da}\left(\frac{16\,b_2^2\,{\rm K}(k)^2}{L^2}\right) \\
&=&  \left( \frac{da}{dk}\right)^{-1} \left[-\frac{1}{2} \frac{16\,b_2^2\,{\rm K}(k)^2}{L^2}\frac{d}{dk}(J_3) -\frac{J_3}{2}\frac{d}{dk}\left(\frac{16\,b_2^2\,{\rm K}(k)^2}{L^2}\right)\right],
\end{eqnarray*}
where $J_3$ is defined by \eqref{J1J2J3}. It remains to calculate the last two terms in the RHS of $(\ref{QI5})$. In fact, we have
\begin{equation*}
(\varphi, \varphi)_{L_{per}^2}= \int_0^L \varphi(x)^2\; dx= b_0^2 L+2b_0b_2 J_1+b_2^2 J_2,
\end{equation*}
and
\begin{equation*}
(\varphi, \varphi'')_{L_{per}^2}= \int_0^L \varphi''(x) \varphi(x) \; dx = -\int_0^L \varphi'(x)^2\; dx = \frac{-16 \, b_2^2 \,{\rm K}(k)^2}{L^2}J_3.
\end{equation*}
 
On the other hand, since $\mathcal{L}_4$ is a positive operator (see Lemma \ref{lemma-L4}) there exists $\alpha_0>0$ such that
$$
(\mathcal{L}_4\, g, g )_{L_{per}^2} \geq\alpha_0 \|g\|_{L_{per}^2}^2,
$$
where $\alpha_0>0$ can be characterized as the first eigenvalue of the operator $\mathcal{L}_4$ by using the min-max principle. Thus,  it follows that
\begin{equation}\label{limitation-L4-alpha0}
(\mathcal{L}_4^{-1}\, f, f )_{L_{per}^2} \leq \frac{1}{\alpha_0} \|f\|_{L_{per}^2}^2,
\end{equation}
for all $f\in L_{per}^2$. Moreover, if $\xi_0 \in {\rm D}(\mathcal{L}_4)$ is the eigenfunction associated to the eigenvalue $\alpha_0>0$, we obtain since $a<0$ and $b_2<0$ that
\begin{align*}
	\alpha_0 \|\xi_0\|^2_{L_{per}^2} &= (\mathcal{L}_4 \xi_0, \xi_0)_{L_{per}^2}  = (a \xi_0'' + \xi_0 - \xi_0 \varphi, \xi_0)_{L_{per}^2} \\
	& = -a\|\xi_0'\|^2_{L_{per}^2} + (1-b_0) \|\xi_0\|^2_{L_{per}^2} - b_2 \int_{0}^{L} {\rm cn}^2 \left(\frac{2 {\rm K}(k)}{L} x,k\right) \xi_0(x)^2 \; dx \\
	& \geq (1-b_0) \|\xi_0\|^2_{L_{per}^2},
\end{align*}
so that $
	\alpha_0 \geq (1- b_0) > 0.
$ Thus, we obtain by \eqref{limitation-L4-alpha0}  
\begin{equation}\label{limitation-L4-b0}
	(\mathcal{L}_4^{-1}\, f, f )_{L_{per}^2} \leq \frac{1}{\alpha_0} \|f\|_{L_{per}^2}^2 \leq \frac{1}{1-b_0} \|f\|_{L_{per}^2}^2.
\end{equation}

Next, by \eqref{QI2} we see that $\mathcal{L}_4\varphi = a \varphi'' + \varphi - \varphi^2 = -2\varphi^2 = 2a\varphi'' + 2 \varphi$ and since $\mathcal{L}_4^{-1}$ is invertible, it follows that $\mathcal{L}_4^{-1}(a \varphi''+\varphi)=\tfrac{\varphi}{2}$. Letting $\varphi-b\varphi''=\tfrac{b}{a}(a\varphi''+\varphi)+\big(1+\tfrac{b}{a}\big)\varphi$, using the notation $f_0 = \varphi-b\varphi''$ and the inequality  \eqref{limitation-L4-b0} for $f=\varphi$, we see that
\begin{equation}\label{QI7}\begin{array}{llll} 
(\mathcal{L}_4^{-1} f_0, f_0)_{L^2_{per}} &=\displaystyle{ \frac{b^2}{2a}(\varphi, \varphi'')_{L_{per}^2}-\left(2-\frac{b}{2a}\right)\frac{b}{2a}\|\varphi\|^2_{L_{per}^2}} +\displaystyle\left(1+\frac{b}{a}\right)^2( \mathcal{L}_4^{-1}\varphi,\varphi)_{L_{per}^2}           \\
   &\leq  \displaystyle   -\frac{b^2}{2a}||\varphi'||_{L_{per}^2}^2+\left(\left(1+\frac{b}{a}\right)^2\left(\frac{1}{1-b_0}\right)-\left(2-\frac{b}{2a}\right)\frac{b}{2a}\right) ||\varphi||_{L_{per}^2}^2.
\end{array}\end{equation}

By considering $\mathcal{J}_3=-\frac{b^2}{2a}||\varphi'||_{L_{per}^2}^2+\left(\left(1+\frac{b}{a}\right)^2\left(\frac{1}{1-b_0}\right)-\left(2-\frac{b}{2a}\right)\frac{b}{2a}\right) ||\varphi||_{L_{per}^2}^2$, we obtain from \eqref{QI}, \eqref{QI5} and \eqref{QI7} that
\begin{eqnarray}\label{quantityIJ}
I=\big( \mathcal{L}\Phi, \Phi \big)_{L_{per}^2\times L_{per}^2}\leq \frac{8}{3}\mathcal{J}_1+\frac{1}{3}\mathcal{J}_3.
\end{eqnarray}

In Figure \ref{figure6}, we put forward some graphics concerning $\frac{8}{3}\mathcal{J}_1+\frac{1}{3}\mathcal{J}_3$ as a function of $k \in (0,1)$ for fixed values of the period $L>0$. We see that $\frac{8}{3}\mathcal{J}_1+\frac{1}{3}\mathcal{J}_3<0$, so that $I<0$ as requested in assumption (H2). Important to mention that the graphics have the behaviour for $B=\pm\sqrt{2}$.

\begin{figure}[!h]
	\subfigure[$L=1$]{
		\includegraphics[width=4.95cm,height=3.6cm]{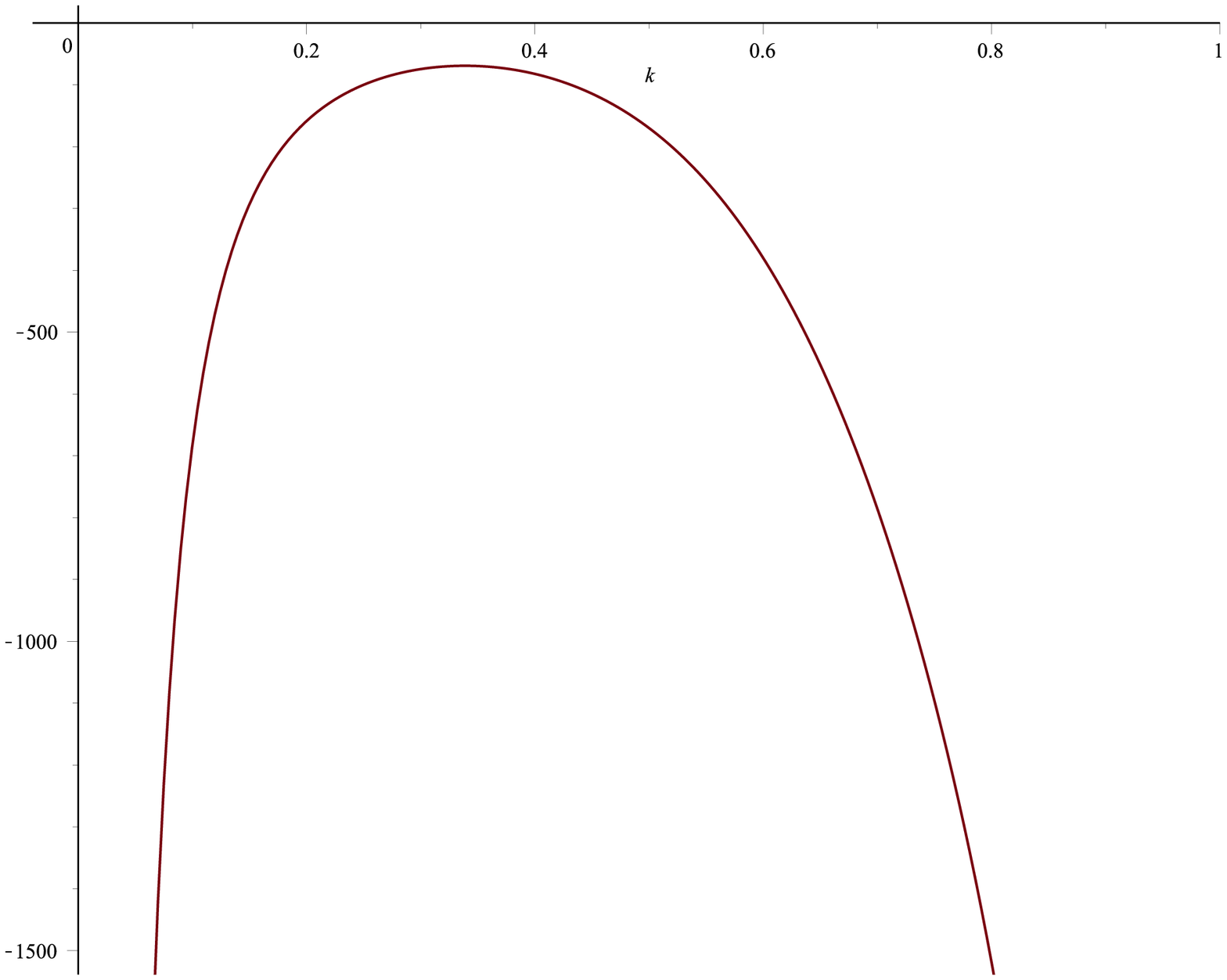}}
	\quad
	\subfigure[$L=2\pi$]{
		\includegraphics[width=4.95cm,height=3.6cm]{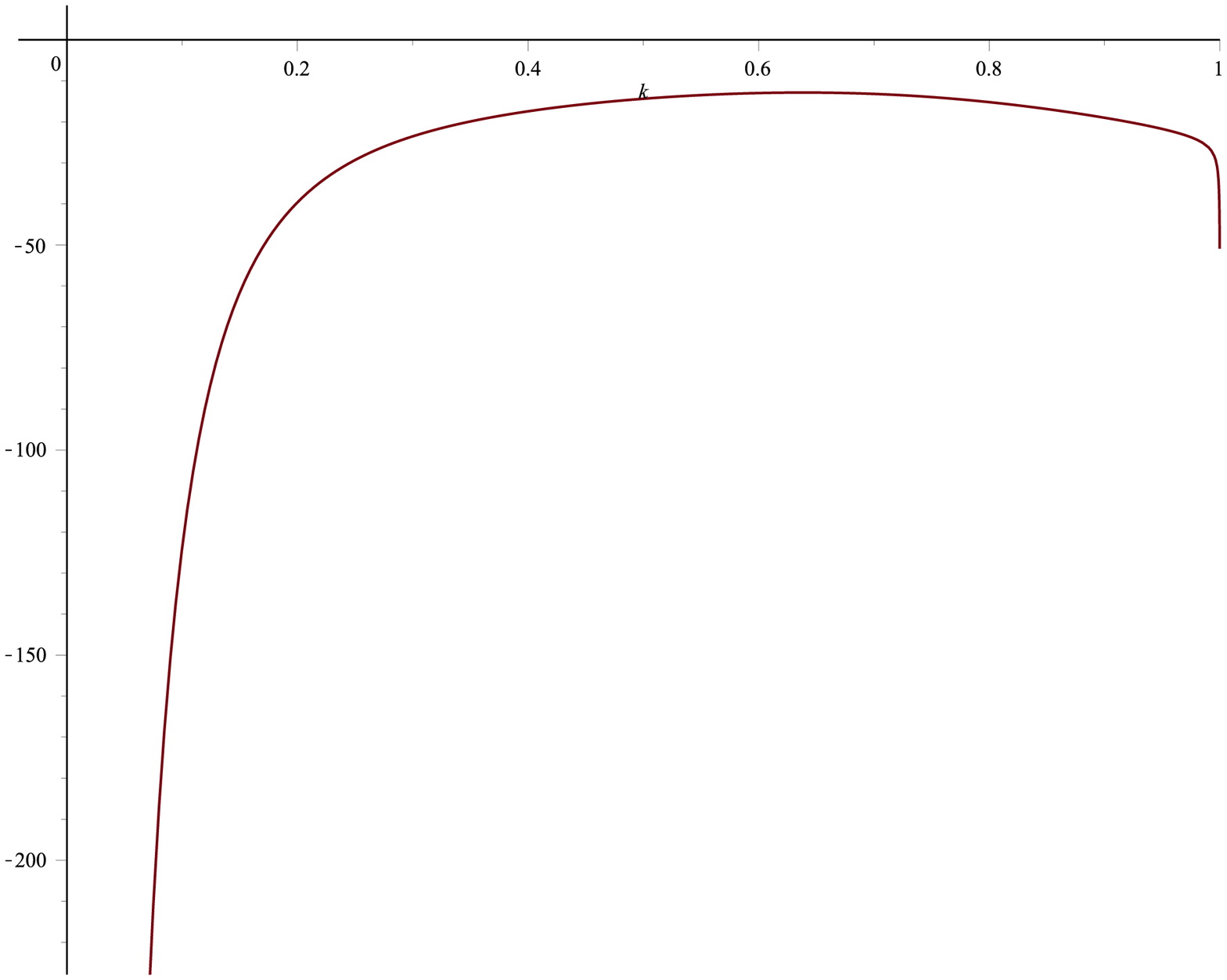}}
	\quad
	\subfigure[$L=50$]{
		\includegraphics[width=4.95cm,height=3.6cm]{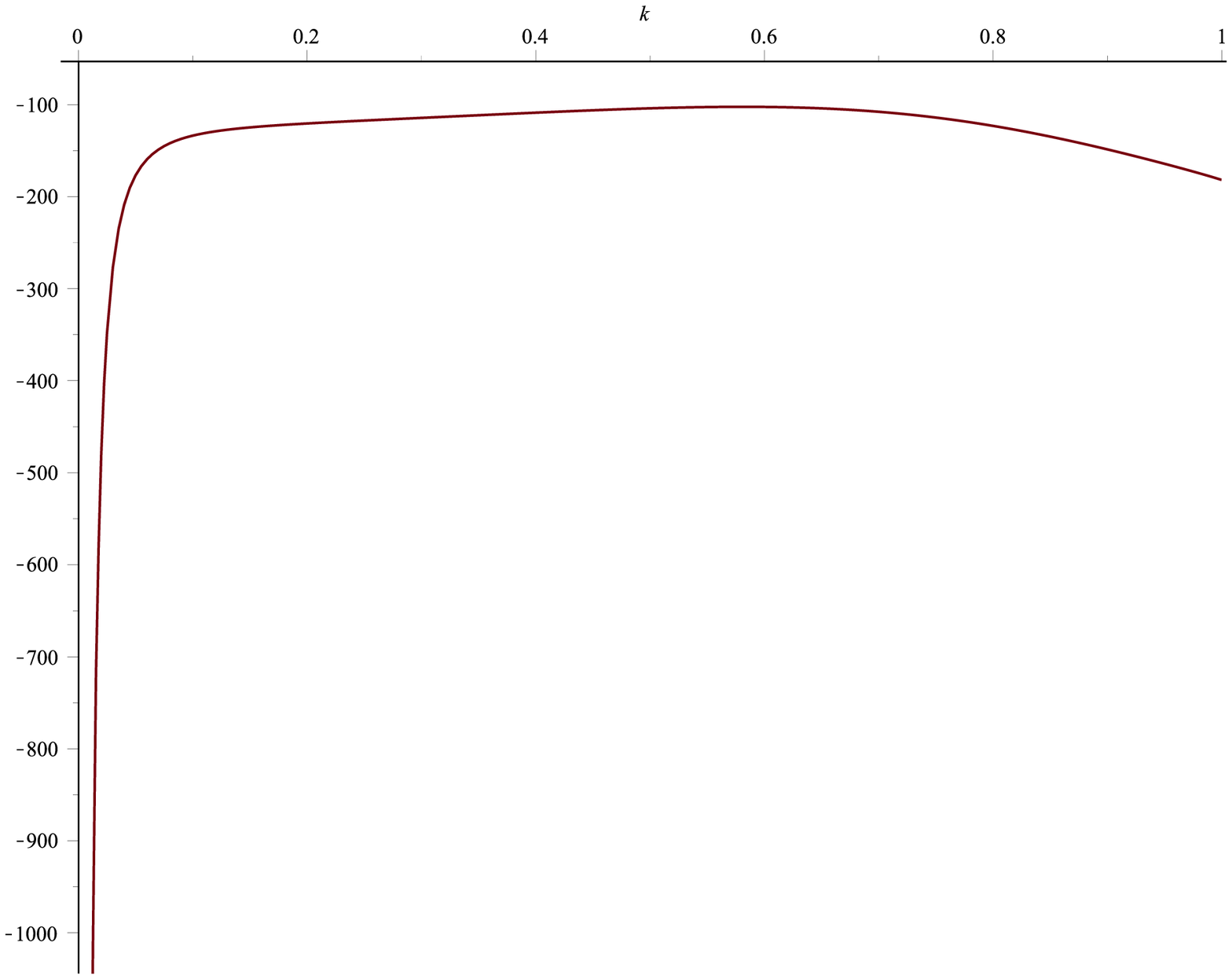}}
	\quad
	\subfigure[$L=100$]{
		\includegraphics[width=4.95cm,height=3.6cm]{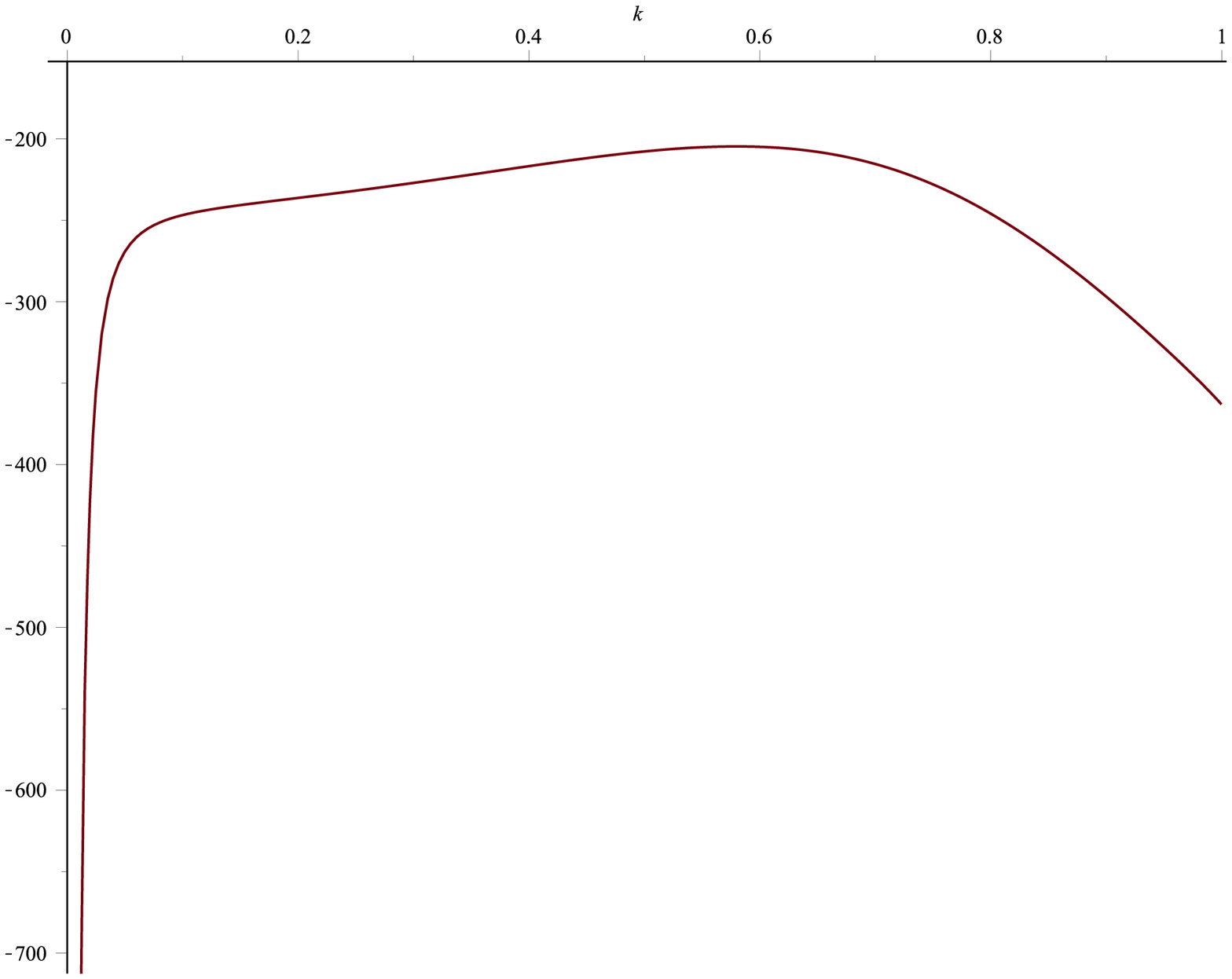}}
	\caption{Graphic of the quantity $\frac{8}{3}\mathcal{J}_1+\frac{1}{3}\mathcal{J}_3$ as a function on $k \in (0, 1)$ %$k \in \big(0, \sfrac{1}{\sqrt{2}}\big)$
	  for some fixed values of $L>0$.}
	\label{figure6}
\end{figure}
\begin{flushright}
	$\blacksquare$
\end{flushright}

\newpage

\section*{Acknowledgments}
G. de Loreno and G. E. B. Moraes are supported by Coordenação de Aperfeiçoamento de Pessoal de Nível Superior (CAPES)/Brazil - Finance code 001. F. Natali is partially supported by Funda\c c\~ao Arauc\'aria/Brazil (grant 002/2017), CNPq/Brazil (grant 303907/2021-5) and CAPES MathAmSud (grant 88881.520205/2020-01).

\end{document}